\documentclass[a4paper, reqno, 11pt]{amsart}
\usepackage{color}
\makeatletter

\@addtoreset{equation}{section}
\makeatother
\usepackage{setspace}
\usepackage{xcolor}
\usepackage{here}
\usepackage{graphicx}
\usepackage{float}
\usepackage{hyperref}
\usepackage[T1]{fontenc}
\usepackage{amsmath}
\usepackage{amssymb}
\usepackage{latexsym}
\usepackage{amsthm}
\usepackage{hyperref}
\newtheorem{Thm}{Theorem}[section]

\newtheorem{Lem}[Thm]{Lemma}
\newtheorem{Prop}[Thm]{Proposition}

\theoremstyle{definition}

\newtheorem{Rem}[Thm]{Remark}
\newtheorem{Exa}[Thm]{Example}

\begin{document}

\title[Construction of graph-directed set]{Construction of graph-directed invariant sets of weak contractions on semi-metric spaces}
\author{Kazuki Okamura}
\address{Department of Mathematics, Faculty of Science, Shizuoka University}
\email{okamura.kazuki@shizuoka.ac.jp}
\dedicatory{Dedicated to the 100th anniversary of the birth of J\'anos Acz\'el}
\maketitle

\begin{abstract}
We present a construction of graph-directed invariant sets of weak contractions in the sense of Matkowski-Rus on semi-metric spaces. 
We follow the approach by Bessenyei and P\'enzes, which applies  the Kuratowski noncompactness measure without relying on Blascke's completeness theorem. 
We also establish a relationship between this approach and a generalized de Rham's functional equation indexed by a finite directed graph.
\end{abstract}

\section{Introduction}

Hutchinson's self-similar sets \cite{Hutchinson1981} are canonical examples of fractal sets. 
Such a set is realized as the fixed point of a transformation on a metric space that maps a subset of the space to the union of the images of the subset under a finite family of transformations, known as an iterated function system (IFS). 
The fixed point is called the invariant set or attractor of the IFS. 
\cite{Hutchinson1981} shows that there exists a unique non-empty compact invariant subset if the IFS consists of contractions, 
by applying Banach's fixed point theorem to the space of non-empty compact subsets endowed with the Hausdorff-Pompeiu metric.

We now assume that a finite directed multigraph allowing self-loops and a metric space are given, with a transformation on the metric space assigned to each directed edge. 
This setting allows us to define a natural analogue of Hutchinson's self-similar sets, known as graph-directed fractal sets. 
These are described as vectors of non-empty compact subsets and can be realized as fixed points of certain transformations on product spaces. 
See \eqref{eq:def-T} below for the precise definition of the transformation. 
The fixed point is called a graph-directed invariant set, which is a notion first introduced by Mauldin and Williams \cite{MauldinWilliams1988} as a generalization of Hutchinson's self-similar sets. 
When the underlying graph consists of a single vertex, the framework reduces to the classical IFS. 
Similar to self-similar sets, graph-directed invariant sets and related topics have been studied extensively. 
For further details, see the book by Mauldin and Urba\'nski \cite{MauldinUrbanski2003}. 

In this paper, we construct an invariant vector of non-empty compact subsets for graph-directed IFSs consisting of weak contractions on semi-metric spaces. 
Semi-metric is obtained by omitting the triangle inequality from Fr\'echet's definition of metric. 
We work specifically with regular semi-metrics introduced by Bessenyei and P\'ales \cite{BessenyeiPales2017}. 
Informally, the local properties of semi-metric spaces are similar to those of metric spaces, on the other hand, a semi-metric can differ significantly in its global structure. 
We assign weak contractions in the sense of Matkowski-Rus \cite{Matkowski1975,Rus2001} as the components of the graph-directed IFS. 
Since weak contractions are less restrictive than standard contractions, a more careful analysis is required. 
Our arguments depend on the fixed point theorem provided in \cite[Theorem 1]{BessenyeiPales2017}.

Our approach follows the method introduced by Bessenyei and P\'enzes \cite{BessenyeiPenzes2020,BessenyeiPenzes2021,BessenyeiPenzes2022}, which differs from Hutchinson's approach. 
The method of Bessenyei and P\'enzes avoids the use of Banach's fixed point theorem which depends on Blaschke's completeness theorem  \cite{Blaschke1949}, and instead applies Kuratowski's measure of noncompactness \cite{Kuratowski1930}. 
This enables us to approximate the unique invariant non-empty compact subset by an explicitly constructible, increasing sequence of finite sets. 
Furthermore, it clearly demonstrates that the invariant subset depends on the topology, rather than the metric.

Our main result (Theorem \ref{thm:main-construct-inv-GDM}) extends the result of \cite{BessenyeiPenzes2022} on invariant sets of weakly contractive IFSs in semi-metric spaces to the graph-directed framework. 
The main technical novelty lies in the construction of a subinvariant vector of singletons (Lemma \ref{lem:1-subinv-exist}). 
This construction aligns naturally with a generalization of de Rham's functional equations \cite{Rham1957}, which is extended here to the graph-directed setting. 
We prove the existence and continuity of the solution. 
A similar generalization was considered by Muramoto and Sekiguchi \cite{MuramotoSekiguchi2020}.
However, our formulation via the invariant vector and conjugate equations offers a simpler and more transparent framework.

\subsection{Framework and main result}

Let $X$ be a non-empty set. 
We say that a map $d : X^2 \to [0, \infty)$ is a {\it semi-metric} on $X$ if $d(x,y) = 0$ if and only if $x=y$, and $d(x,y) = d(y,x)$ for every $x, y \in X$.  
Let $B(x,r) := \{y \in X | d(x,y) < r\}$ for $x \in X$ and $r > 0$. 
We say that a subset $U$ of $X$ is an open set if 
for every $x \in X$, there exists $r > 0$ such that $B(x,r) \subset U$. 
This defines a topology on $X$. 
We remark that $B(x,r)$ can be a non-open set.  
A closed set is a complement of an open set. 
For $H \subset X$, we denote the closure of $H$ by $\overline{H}$. 
We say that a sequence $(x_n)_n$ in $X$ is convergent if  $\lim_{n \to \infty} d(x_n, x) = 0$ for some $x \in X$, and it is a Cauchy sequence if $\lim_{n, m \to \infty} d(x_n, x_m) = 0$. 
We say that $(X,d)$ is complete if every Cauchy sequence is convergent.

We introduce a notion which plays the role of the classical triangle inequality. 
For $u, v \in [0,\infty)$, let 
\[ \mathcal{I}(u,v) := \left\{(x,y) \in X^2 \middle|  \textup{ there exists }   z \in X \textup{ such that } d(x,z) \le u, d(y,z) \le v\right\}, \]
and, 
\[ \Phi_d (u,v) := \sup\left\{d(x,y) | (x,y) \in \mathcal{I}(u,v) \right\} \in [0,\infty]. \]
For every $x \in X$, $(x,x) \in \mathcal{I}(u,v)$, and in particular, $\mathcal{I}(u,v) \ne \emptyset$. 
We call $\Phi_d$ the {\it basic triangle function}. 

We say that a semi-metric space $(X,d)$ satisfies the {\it regular} property for the basic triangle function if $\Phi_d$ is continuous at $(0,0)$, and, it satisfies the  {\it normal} property for the basic triangle function if $\Phi_d (u,v) \ne \infty$ for $u, v \in [0,\infty)$. 
Every metric space is a semi-metric space satisfying the regular and normal properties for the basic triangle function. 
The regular property was introduced in \cite{BessenyeiPales2017} and it plays the role of the classical triangle inequality. 
Informally, we can say that a regular metric satisfies a kind of  triangle inequality in ``small'' scales. 
The normal property was introduced in \cite{BessenyeiPenzes2022} and it is important when we deal with bounded subsets.  
The regular or normal property is different from the notion of the regular or normal topological space (${\bf T}_3$ space or ${\bf T}_4$ space). 

{\it Hereafter in this subsection, we assume that $(X,d)$ is a complete semi-metric space satisfying the regular and normal properties for the basic triangle function. }

We adopt Matkowski-Rus's definition\footnote{This is sometimes called the Matkowski contraction, however, the fixed point theorem for this weak contraction was independently proved by Matkowski and Rus in 1975.} of weak contractions \cite{Matkowski1975,Rus2001}.  
We say that a map $\phi : [0,\infty) \to [0,\infty)$ is a {\it comparison function} if 
$\phi$ is increasing, and 
$\lim_{n \to \infty} \phi^n (t) = 0$ for every $t > 0$. 
Let $\phi$ be a comparison function. 
We say that a map $T : X \to X$ is a {\it $\phi$-contraction} if 
$d(Tx, Ty) \le \phi(d(x,y))$ for every $x, y \in X$.

We state the framework of graph-directed IFSs as in \cite[Section 3.1]{Falconer1997}.  
Let $q \ge 1$ and $V := \{1, \cdots, q\}$. 
For $i, j \in V$, 
let $E_{i,j}$ be a finite set of directed edges from $i$ to $j$. 
We remark that $i=j$ is allowed. 
We allow multiple edges and self-loops in order to make the setting here a generalization of the self-similar set, which corresponds to the case that $q=1$.  
We assume that for every $i \in V$, there exists $j \in V$ such that $E_{i,j} \ne \emptyset$. 
This assumption is weaker than the transitivity condition. 
Assume that for each edge $e \in E_{i,j}$,  a {\it right-continuous} comparison function $\phi_e$ and  a $\phi_e$-contraction $f_e : X  \to X$ are given. 
The case that $q=1$ gives the framework of Hutchinson's self-similar sets. 

Let $\mathcal{P}(X)$ be the set of non-empty subsets of $X$. 
Hereafter, we need to deal with product spaces $X^q$ and $\mathcal{P}(X)^q$. 
Elements in these product spaces are written in bold face. 

Let 
\begin{equation}\label{eq:def-T} 
T({\bf H}) := \left(\bigcup_{j = 1}^{q} \bigcup_{e \in E_{i,j}} f_e (H_j)\right)_{i = 1}^{q}, \  {\bf H} = (H_1, \dots, H_q) \in \mathcal{P}(X)^q.  
\end{equation}

We denote ${\bf A} \subset {\bf B}$
if $A_i \subset B_i$ for each $i \in V$, where $ {\bf A} = (A_1, \dots, A_q), {\bf B} = (B_1, \dots, B_q) \in \mathcal{P}(X)^q$. 
We say that ${\bf H} \in \mathcal{P}(X)^q$ is {\it $T$-subinvariant} if ${\bf H} \subset T({\bf H})$ and {\it $T$-invariant} if ${\bf H} = T({\bf H})$. 

For a family $({\bf H}_n)_n$ in $\mathcal{P}(X)^q$, 
let $\cup_n {\bf H}_n := \left( \cup_{n} H_{n,1}, \dots, \cup_{n} H_{n,q} \right)$, where ${\bf H}_n =: \left(H_{n,1}, \dots, H_{n,q}\right)$ for each $n$. 
Let ${\bf \overline{H}} := \left(\overline{H_1}, \dots, \overline{H_q}\right)$, where ${\bf H} =: (H_1, \dots, H_q) \in \mathcal{P}(X)^q$. 
Let $\mathcal{K}(X)$ be the set of all elements of $\mathcal{P}(X)$ which are compact. 

The following is our main result. 
\begin{Thm}\label{thm:main-construct-inv-GDM}
There exists a unique ${\bf H} \in \mathcal{K}(X)^q$ which is $T$-invariant. 
Furthermore, there exists  ${\bf H}_0 \in  \mathcal{P}(X)^q$ such that ${\bf H}_0$ is a product of singletons of $X$ and is $T$-subinvariant, and ${\bf H} =  \overline{\cup_{n} T^n ({\bf H}_0)}$. 
\end{Thm}

The existence and uniqueness of the invariant set will be shown following Hutchinson's approach \cite{Hutchinson1981} and extending a result by Kocsis and P\'ales \cite[Theorem 17]{KocsisPales2022} to the graph-directed framework.   
However, the approximation of the invariant set by the iterated images of a finite set is important in some examples including the functional equation in Section \ref{sec:fe}.

The rest of this paper is organized as follows. 
In Section \ref{sec:prelim}, we give some preliminary results needed for the proof. 
Section \ref{sec:proof} is devoted to the proof of Theorem \ref{thm:main-construct-inv-GDM}. 
In Section \ref{sec:trans-semi-met}, we consider transforms of semi-metrics and give an example of the graph-directed invariant set.  
We deal with  graph-directed de Rham's functional equations in Section \ref{sec:fe}. 

\section{Preliminaries}\label{sec:prelim}

\subsection{Symmetric pre-metric space}\label{subsec:premetric}

In the sequel, we need to deal with a symmetric, pre-metric space. 
Let $X$ be a non-empty set. 
We say that a map $d : X^2 \to [0, \infty)$ is a {\it symmetric pre-metric} on $X$ if 
$d(x,x) = 0$ for every $x \in X$, and $d(x,y) = d(y,x)$ for every $x, y \in X$.  
A symmetric pre-metric  $d : X^2 \to [0, \infty)$ is 
a semi-metric on $X$ if and only if the positivity holds, specifically, the property that $d(x,y) = 0$ implies $x=y$. 
For a symmetric pre-metric space, 
we introduce a topology in the same manner as in the semi-metric case. 
The notions of convergence and Cauchy sequence are also defined in the same manner. 
However, the limit of a convergent sequence may not be unique. 
We can also define the basic triangle function in the same manner.  
The regular and normal properties are also defined  in the same manner.

We now state some consequences of the normal property. 
We say that a subset $H$ of $X$ is {\it bounded} if there exist $x_0 \in X$ and $r_0 > 0$ such that $H \subset B(x_0, r_0)$. 
Let $\textup{diam}(H) := \sup\{d(x,y) | x, y \in A\}$ for $H \subset X$, $H \ne \emptyset$, and we call it the {\it diameter} of $H$. 

\begin{Lem}[{\cite[Proposition 3]{BessenyeiPenzes2022}}]\label{lem:bounded-basic}
Let $(X,d)$ be a normal semi-metric space and $H$ be a non-empty subset of $X$. 
Then, $H$ is bounded if and only if $\textup{diam}(H) < \infty$. 
\end{Lem}

For $H \subset X$, $H \ne \emptyset$, 
let $H(r) := \cup_{x \in H} B(x,r)$, which is the $r$-neighborhood of $H$. 

\begin{Lem}\label{lem:conseq-normal-semi}
Let $(X,d)$ be  a normal semi-metric space. 
Then,\\
(i) Every Cauchy sequence on $X$ is bounded. \\
(ii) If $H$ is a non-empty bounded subset of $H$ and $r > 0$, 
then, $H(r)$ is also bounded.  \\
(iii) Let $H_1$ and $H_2$ be bounded subsets of $H$. 
Then, $H_1 \cup H_2$ is bounded. 
\end{Lem}

\begin{proof}
(i) Let $(x_n)_{n=0}^{\infty}$ be a Cauchy sequence. 
Since $d(x_i, x_j) \le \Phi_d (d(x_i, x_0), d(x_j, x_0))$, for the coordinate-wise increasing property of $\Phi_d$, and the normal property for $(X,d)$,  
it suffices to show that $\sup_i d(x_i, x_0) < \infty$. 
By the assumption, there exists $N$ such that for every $m, n \ge N$, $d(x_m, x_n) \le 1$.
Then, for every $i \ge N$, 
$d(x_i, x_0) \le \Phi_d (d(x_i, x_N), d(x_N, x_0)) \le \Phi_d (1, d(x_N, x_0))$. 
Hence, $\sup_i d(x_i, x_0) < \infty$. 

(ii) By Lemma \ref{lem:bounded-basic}, there exist $x_0 \in X$ and $r_0 > 0$ such that $H \subset B(x_0, r_0)$. 
Let $y \in H(r)$. 
Then, there exists $z \in H$ such that $d(y,z) < r$. 
Hence, by  the coordinate-wise increasing property of $\Phi_d$, 
$d(x_0, y) \le \Phi_d (d(x_0,z), d(z,y)) \le \Phi_d (r_0, r)$. 
By the normal property for $(X,d)$,  $\Phi_d (r_0, r) < \infty$. 
Hence, $y \in B(x_0, \Phi_d (r_0, r)+1)$. 
Thus, $H(r)$ is bounded. 

(iii) If $H_1$ or $H_2$ is empty, then, the assertion clearly holds.
Assume $H_1$ and $H_2$ are both non-empty. 
Let $x_1 \in H_1$ and $x_2 \in H_2$. 
By the assumption, there exist $r_1, r_2 > 0$ such that $H_i \subset B(x_i, r_i)$, $i=1,2$. 
Let $y \in H_2$. 
Then, 
$d(x_1, y) \le \Phi_d (d(x_1, x_2), d(x_2, y)) \le \Phi_d (d(x_1, x_2), r_2) < \infty$. 
Hence, $y \in B(x_1, \Phi_d (d(x_1, x_2), r_2)+1)$. 
Thus, $H_1 \cup H_2 \subset B(x_1,  r_1 + \Phi_d (d(x_1, x_2), r_2)+1)$. 
This means that $H_1 \cup H_2$ is bounded. 
\end{proof}

\cite{ChrzaszczJachymskiTurobos2018} gives a detailed overview of semi-metric spaces. 

\subsection{Contraction and fixed point}\label{subsec:contra-fix}

\begin{Lem}\label{lem:comparison-basic}
(i) If $\phi$ is a comparison function, then, $\phi(0) = 0$ and  $\phi(t) < t$ for every $t > 0$. \\
(ii) If $\phi_1$ and $\phi_2$ are right-continuous comparison functions and $\phi := \max\{\phi_1, \phi_2\}$, then, $\phi$ is also a right-continuous comparison function. 
\end{Lem}

\begin{proof}
(i) Assume that $\phi(t) \ge t > 0$. 
Since $\phi$ is increasing, $(\phi^n (t))_n$ is increasing. 
Since $\lim_{n \to \infty} \phi^n (t) = 0$, $\phi^n (t) = 0$ for every $n$. 
Hence, $t = 0$. 
This contradicts $t > 0$. 
Since $\phi$ is increasing, $\phi(0) \le \phi(t) < t$ for $t > 0$. 
Hence, $\phi(0) = 0$. 

(ii) It is easy to see that $\phi$ is increasing, $\phi(0) = 0$ and $\phi(t) < t$ for $t > 0$. 
Since $\phi = (\phi_1 + \phi_2 + |\phi_1 - \phi_2|)/2$, $\phi$ is right-continuous. 
Then, for $t > 0$, $(\phi^n (t))_n$ is decreasing with respect to $n$. 
Hence, the limit $t_0 := \lim_{n \to \infty} \phi^n (t)$ exists. 
Since $\phi$ is right-continuous, $\phi(t_0) = t_0$. 
Hence, $t_0 = 0$. 
\end{proof}

Jachymski  \cite{Jachymski1997} gave useful equivalent conditions for Matkowski-Rus's weak contraction.  
See Berlinde, Petrusel and Rus \cite{BerindePetruselRus2023} for terminology in fixed point theory. 
Le\'sniak, Snigireva and Strobin's survey \cite{LesniakSnigirevaStrobin2020} contains various results for weakly contractive IFSs.

We now state some consequences of the regular property. 

\begin{Thm}[Bessenyei and P\'ales {\cite[Theorem 1]{BessenyeiPales2017}} ]\label{thm:Bessenyei-Pales-FP}
Let $(X,d)$ be a complete semi-metric space. 
Let $\phi$ be a comparison function. 
Let $T : X \to X$ be a $\phi$-contraction. 
Then, there exists a unique fixed point of $T$. 
\end{Thm}

We give a proof of this assertion, because in the sequel we use the following argument to establish an assertion. 
We adopt the strategy of the proof of \cite[Theorem 1]{BessenyeiPales2017} with minor changes.

\begin{proof}
We show the uniqueness. 
Assume that $Tx=x$ and $Ty=y$. 
Since $T$ is a $\phi$-contraction, 
$d(x,y) = d(Tx, Ty) \le \phi(d(x,y))$. 
If $d(x,y) > 0$, then, $d(x,y) \le \phi(d(x,y)) < d(x,y)$ and this is a contradiction. 
Hence, $d(x,y) = 0$. 
By the positivity of the semi-metric $d$, 
$x=y$. 

We show the existence. 
Take a point $x_0 \in X$. 
Let $x_n := T^n x_0$ for $n \ge 1$. 
We show that $(x_n)_n$ is a Cauchy sequence, that is, $\lim_{m, n \to \infty} d(x_n, x_m) = 0$. 

Let $\epsilon > 0$. 
Then, there exists $\delta(\epsilon) \in (0,\epsilon)$ such that $\Phi(u,v) < \epsilon$ for every $u, v \in [0,\delta(\epsilon)]$. 

For every $n \ge 0$, 
$d(x_{n+1}, x_{n}) \le \phi^n (d(x_1, x_0))$. 
By the assumption for $\phi$, 
$\lim_{n \to \infty} d(x_{n+1}, x_{n}) = 0$. 
Then, there exists $n(\epsilon)$ such that $\phi^{n(\epsilon)}(\epsilon) < \delta(\epsilon)$. 
There also exists $N(\epsilon)$ such that $\phi^{N(\epsilon)}\left(\max_{1 \le i \le n(\epsilon)} d(x_{i}, x_0)\right) < \delta(\epsilon)$. 

Let $y \in B(x_{N(\epsilon)}, \epsilon)$. 
Then, 
\[ d\left(T^{n(\epsilon)} y, x_{N(\epsilon)}\right) \le \Phi_d \left( d(T^{n(\epsilon)} y, x_{N(\epsilon) + n(\epsilon)}), d(x_{N(\epsilon) + n(\epsilon)}, x_{N(\epsilon)}) \right)\]
\[  \le \Phi_d \left(\phi^{n(\epsilon)}(d(y, x_{N(\epsilon)})),  \phi^{N(\epsilon)}(d(x_{n(\epsilon)}, x_0))\right)   \le \Phi_d (\delta(\epsilon), \delta(\epsilon)) < \epsilon. \]
Hence, $T^{n(\epsilon)} (B(x_{N(\epsilon)}, \epsilon) ) \subset B(x_{N(\epsilon)}, \epsilon)$. 
Thus, for every $k \ge n(\epsilon)$, 
$$x_{k+N(\epsilon)} \in T^k  (B(x_{N(\epsilon)}, \epsilon) ) \subset \bigcup_{i=0}^{n(\epsilon)-1} T^i (B(x_{N(\epsilon)}, \epsilon)).$$

In the same manner, for $0 \le i \le n(\epsilon)-1$, 
\[ d\left(T^{i} y, x_{N(\epsilon)}\right) \le \Phi_d \left( d(T^{i} y, x_{N(\epsilon) + i}), d(x_{N(\epsilon) + i}, x_{N(\epsilon)}) \right)\]
\[  \le \Phi_d \left(\phi^{i}(d(y, x_{N(\epsilon)})),  \phi^{N(\epsilon)}(d(x_{i}, x_0))\right)   \le \Phi_d (\epsilon, \epsilon). \]
Hence, 
\[ \bigcup_{i=0}^{n(\epsilon)-1} T^i (B(x_{N(\epsilon)}, \epsilon)) \subset B\left(x_{N(\epsilon)}, \Phi_d (\epsilon, \epsilon) \right). \]

So, for $i, j \ge N(\epsilon) + n(\epsilon)$, 
\[ d(x_i, x_j) \le \Phi_d \left(d(x_i, x_{N(\epsilon)}), d(x_j, x_{N(\epsilon)})\right) \le \Phi_d (\Phi_d (\epsilon, \epsilon), \Phi_d (\epsilon, \epsilon)) \to 0, \epsilon \to +0. \]
Thus we see that $(x_n)_n$ is a Cauchy sequence.  

By the completeness, we have the limit $x_{\infty} \in X$ such that $\lim_{n \to \infty} d(x_n, x_{\infty}) = 0$. 
Therefore, 
\[ d(Tx_{\infty}, x_{\infty}) \le \Phi_d (d(Tx_{\infty}, x_{n+1}), d(x_{n+1}, x_{\infty})) = \Phi_d (d(Tx_{\infty}, Tx_{n}), d(x_{n+1}, x_{\infty})) \]
\[ \le \Phi_d (d(x_n, x_{\infty}), d(x_{n+1}, x_{\infty})) \to 0, \ n \to \infty. \]

Hence, $d(Tx_{\infty}, x_{\infty}) = 0$. 
So, by the positivity of the semi-metric $d$, $Tx_{\infty} = x_{\infty}$. 
\end{proof}

We remark that we do not need the positivity of $d$, specifically, $d(x,y) = 0 \Rightarrow x = y$, which is not used in the proof. 
The definition of $\phi$-contraction is also valid for symmetric, pre-metric spaces. 
By Lemma \ref{lem:conseq-normal-semi} and the proof of Theorem \ref{thm:Bessenyei-Pales-FP}, 
we see that 
\begin{Lem}\label{lem:Cauchy-bounded-normal-regular}
Let $(X,d)$ be a symmetric pre-metric space satisfying the regular and normal properties. 
Let $\phi$ be a comparison function. 
Let a map $T : X  \to X$ be a $\phi$-contraction. 
Then, for every $x \in X$, a sequence $(T^n x)_n$ is Cauchy and bounded. 
\end{Lem}

\subsection{Kuratowski measure of non-compactness}\label{subsec:Kuratowski}

Let $(X, d)$ be a semi-metric space. 
For $H \subset X$, let 
\[ \chi(H) := \inf\left\{r > 0 | \textup{ there exist }  n \ge 1,  x_1, \dots, x_n \in X \textup{ such that  } H \subset \cup_{i=1}^{n} B(x_i, r) \right\}, \]
where we let $\inf \emptyset := \infty$. 
We call this the {\it Kuratowski measure of non-compactness} of $H$. 
For metric spaces, this notion was introduced in \cite{Kuratowski1930}, 
and, the monograph by Granas and Dugundji \cite{GranasDugundji2003} gives properties of the Kuratowski measure of non-compactness.

By definition, $\chi(H) < \infty$ if $H$ is bounded. 
It holds that 
\begin{equation}\label{eq:Kuratowski-union}
\chi(H_1 \cup H_2) \le \max\{\chi(H_1), \chi(H_2)\}, \ \ H_1, H_2 \subset X.
\end{equation} 

\begin{Lem}[Cf.{\cite[Lemma 4]{BessenyeiPenzes2022}}]\label{lem:condense}
Let $\phi$ be a right-continuous comparison function. 
Let $f$ be a  $\phi$-contraction on $X$. 
Let $H \subset X$ and $0 < \chi(H) < \infty$. 
Then, $\chi(f(H)) < \chi(H)$. 
\end{Lem}

This assertion means that a $\phi$-contraction is a {\it condensing} map. 

\begin{proof}
If $\chi(H) < r$, then, $H \subset \cup_{i=1}^{n} B(x_i, r)$. 
Then, for every $\epsilon > 0$, 
$f(H) \subset \cup_{i=1}^{n} B(f(x_i), \phi(r)+\epsilon)$. 
Hence, $\chi(f(H)) \le \phi(r)$. 
Since $\phi$ is right-continuous and $\chi(H) > 0$, 
$\chi(f(H)) \le \phi (\chi(H)) < \chi(H)$. 
\end{proof}

We will use the following lemma. 
\begin{Lem}[{\cite[Lemma 2]{BessenyeiPenzes2022}}]\label{lem:zero-Kuratowski-implies-rel-cpt}
Assume that $(X,d)$ is regular. 
Let $H \subset X$. 
Then, $\chi(H) = 0$ if and only if $H$ is a relatively compact subset in $X$. 
\end{Lem}

\subsection{Hausdorff-Pompeiu distance}

Let $(X,d)$ be a normal, regular semi-metric space. 
Let $\mathcal{F}(X)$ be the set of non-empty bounded subsets of $X$. 
We do not require that the subsets are closed or compact. 

Let 
\[ d_{HP}(A,B) := \inf\{r > 0 | A \subset B(r), B \subset A(r) \}, \ A, B \in \mathcal{F}(X), \]
where $\inf \emptyset := \infty$. 
This definition is the same as the definition of Hausdorff-Pompeiu distance \cite{Pompeiu1905, Hausdorff1949} for metric spaces. 

\begin{Lem}\label{lem:HP-normal-and-regular}
(i) $d_{HP}$ is a symmetric pre-metric on $\mathcal{F}(X)$. \\
(ii) $(\mathcal{F}(X), d_{HP})$ satisfies the normal property and the regular property. 
\end{Lem}

We denote the basic triangle function of $d_{HP}$ by $\Phi_{d_{HP}}$. 

\begin{proof}
(i) By definition, $d_{HP}$ is symmetric and non-negative, specifically, $d_{HP}(A,B) = d_{HP}(B,A) \ge 0$. 
Furthermore, $d_{HP}(A,A) = 0$. 

Now it suffices to show that $d_{HP}(A,B) < \infty$. 
By symmetry, it suffices to show that for some $r > 0$, $A \subset B(r)$. 
Since $A, B$ are non-empty, we can take $x_0 \in A$ and $y_0 \in B$. 
Since $A$ is bounded, by Lemma \ref{lem:bounded-basic}, $\textup{diam}(A) < \infty$. 
Let $x \in A$. 
Then,  
$d(x,y_0) \le \Phi_d (d(x,x_0), d(x_0, y_0)) \le \Phi_d (\textup{diam}(A), d(x_0, y_0))$. 
By the normal property for $d$, $r_0 :=  \Phi_d (\textup{diam}(A), d(x_0, y_0)) < \infty$. 
Hence, $A \subset B(y_0, r_0 +1) \subset B(r_0 +1)$. 

(ii) Step 1. We first show that if $A, B, C \in \mathcal{F}(X)$ and $A \subset B(r_1)$ and $B \subset C(r_2)$, then, 
for every $\epsilon > 0$, 
$A \subset C(\Phi_d (r_1, r_2) + \epsilon)$. 

Let $x \in A$. 
Then, there exist $y \in B$ and $z \in C$ such that $d(x,y) < r_1$ and $d(y,z) < r_2$. 
Then, 
$d(x,z) \le \Phi_d (d(x,y), d(y,z)) \le \Phi_d (r_1, r_2)$. 
Hence, $x \in B(z, \Phi_d (r_1, r_2) + \epsilon) \subset C(\Phi_d (r_1, r_2) + \epsilon)$ for every $\epsilon > 0$. 

Step 2. Assume that $A, B, C \in \mathcal{F}(X)$, $r_1 > d_{HP}(A,B)$ and $r_2 > d_{HP}(B,C)$. 
Then, by Step 1, 
$d_{HP}(A,C) \le \Phi_d (r_1, r_2) + \epsilon$ for every $\epsilon > 0$, and hence, $d_{HP}(A,C) \le \Phi_d (r_1, r_2)$. 
Thus, 
for every $u, v \ge 0, \epsilon > 0$, 
$\Phi_{d_{HP}}(u, v) \le \Phi_d (u+\epsilon, v+\epsilon)$. 
By the normal property for $d$, $\Phi_{d_{HP}} (u, v)  \le \Phi_d (u+\epsilon, v+\epsilon) < \infty$.  
Hence $d_{HP}$ satisfies the normal property. 

Step 3. Let $\epsilon > 0$. 
Since $d$ satisfies the regular property, there exists $\delta > 0$ such that for every $u, v \in [0,\delta]$, $\Phi_d (u,v) < \epsilon$. 
Hence, 
for every $u, v \in [0,\delta/2]$, 
$\Phi_{d_{HP}}(u, v) \le \Phi_d (u+\delta/2, v+\delta/2) < \epsilon$. 
Thus we see that $\Phi_{d_{HP}}$ is also continuous at $(0,0)$. 
Hence $d_{HP}$ satisfies the regular property. 
\end{proof}

\begin{Lem}\label{lem:weak-contraction-semimet-HP}
Let $(X,d)$ be a semi-metric space. 
Let $\phi_i, 1 \le i \le n$, be right-continuous comparison functions on $X$.  
Let $\phi := \max_{1 \le i \le n} \phi_i$. 
Let $f_i$ be a $\phi_i$-contraction on $X$.  
Then, $\bigcup_{i=1}^{n} f_i ( \mathcal{F}(X)) \subset  \mathcal{F}(X)$,  
\[ d_{HP}\left( \bigcup_{i=1}^{n} f_i (A), \bigcup_{i=1}^{n} f_i (B) \right) \le \phi\left( d_{HP}(A, B) \right), \ A, B \in \mathcal{F}(X), \]
and $\phi$ is a right-continuous comparison function. 
\end{Lem} 

\begin{proof}
Let $A \in \mathcal{F}(X)$. 
Then, by Lemma \ref{lem:bounded-basic}, $\textup{diam}(A) < \infty$. 
Let $x, y \in A$. 
Then, 
\begin{equation}\label{eq:each-f-contract}
d(f_i (x), f_i (y)) \le \phi_i (d(x,y)).
\end{equation}
Hence, $\textup{diam}(f_i (A)) \le  \phi_i (\textup{diam}(A)) < \infty$. 
By Lemma \ref{lem:bounded-basic}, $f_i (A) \in \mathcal{F}(X)$. 
By Lemma \ref{lem:conseq-normal-semi} (iii), 
$\bigcup_{i=1}^{n} f_i ( \mathcal{F}(X)) \subset  \mathcal{F}(X)$. 

By the definition of $d_{HP}$, 
\[ d_{HP}\left( \bigcup_{i=1}^{n} f_i (A), \bigcup_{i=1}^{n} f_i (B) \right) \le \max_{1 \le i \le n}  d_{HP}\left( f_i (A), f_i (B) \right). \]
By \eqref{eq:each-f-contract}, if $A \subset B(r)$, then, $f_i (A) \subset f_i (B) (\phi_i (r) + \epsilon)$ for every $\epsilon > 0$.  
Hence, for every $r > d_{HP}(A,B)$, $d_{HP}\left( f_i (A), f_i (B) \right) \le \phi_i (r)$. 
By the right-continuity of $\phi_i$, 
$d_{HP}\left( f_i (A), f_i (B) \right) \le \phi_i ( d_{HP}(A,B) )$. 
Hence, 
\[  \max_{1 \le i \le n}  d_{HP}\left( f_i (A), f_i (B) \right) \le \phi (d_{HP}(A,B)). \]
By Lemma \ref{lem:comparison-basic}, $\phi$ is a right-continuous comparison function. 
\end{proof}

\begin{Rem}\label{rem:closed}
Even if $A$ is closed, we do not assure that $f(A)$ is closed. 
Let $X = (0,\infty) \setminus \{1\}$ be equipped with the topology induced by the Euclidean metric.   
Let $A := (0,1) \subset X$. 
Then, it is a closed subset of $X$. 
Let $f(x) = (x/2) + 1$. 
Then, $f(X) = (1,\infty) \setminus \{3/2\} \subset X$, and $f(A) = (1,3/2)$, which is not closed in $X$.  
Therefore, $T(A)$ and $T(B)$ may not be closed in the proof in \cite[Lemma 5]{BessenyeiPenzes2021}. 
The argument holds if we replace $T(A)$ and $T(B)$ with $\overline{T(A)}$ and $\overline{T(B)}$, respectively. 
Therefore, we have {\it not} assumed that any element of $\mathcal{F}(X)$ is a closed subset of $X$. 
\end{Rem}

\begin{Lem}\label{lem:HP-is-semi-metric-if-closed}
Let $\overline{\mathcal{F}}(X)$ be the set of all elements of $\mathcal{F}(X)$ which are closed. 
Then, $d_{HP}$ is a semi-metric on $\overline{\mathcal{F}}(X)$. 
\end{Lem}

\begin{proof}
Assume that $d_{HP}(A, B) = 0$. 
Then, for every $n \ge 1$, 
$A \subset B(1/n)$. 
Let $x \in A$. 
Then, for each $n \ge 1$, 
there exists $y_n \in B$ such that $d(x, y_n) < 1/n$. 
By \cite[Theorem 4.2]{ChrzaszczJachymskiTurobos2018}, the limit of every convergent sequence in a closed subset is also in the closed subset\footnote{In \cite[Section 2]{BessenyeiPenzes2022}, it is stated that the neighborhood topology and the sequential topology may differ from each other, but this is not compatible with  \cite[Theorem 4.2]{ChrzaszczJachymskiTurobos2018}.}. 
Since $B$ is closed, $x \in B$. 
Hence, $A \subset B$. 
In the same manner, we see that $B \subset A$. 
\end{proof}

For metric spaces, Blaschke's completeness theorem \cite{Blaschke1949} assures the completeness of $\mathcal{K}(X)$ with respect to the Hausdorff-Pompeiu metric $d_{HP}$.

\subsection{Product space}\label{subsec:product-space}

Hereafter, we need to deal with product spaces $X^q$ and $\mathcal{F}(X)^q$. 
Elements in these product spaces are written in bold face. 

We first deal with the product space of $X$. 
Let 
$$ d_{\infty}\left({\bf x}, {\bf y}\right) := \max_{1 \le i \le q} d(x_i, y_i), \ {\bf x} = (x_1, \dots, x_q), \ {\bf y}  = (y_1, \dots, y_q) \in X^q. $$

\begin{Lem}\label{lem:product-semimetric}
(i) If $(X,d)$ is a semi-metric space, then, $(X^q, d_{\infty})$ is a semi-metric space. \\ 
(ii) If a semi-metric space $(X,d)$ is normal, regular, or complete, 
then, $(X^q, d_{\infty})$ is also normal, regular,  or complete, respectively. 
\end{Lem}

\begin{proof}
It is obvious that $d_{\infty}$ is a semi-metric on $X^q$. 

Let ${\bf x}, {\bf y}, {\bf z} \in X^q$. 
Assume that $d_{\infty}({\bf x}, {\bf y}) \le u$ and $d_{\infty}({\bf y}, {\bf z}) \le v$. 
Let ${\bf x} =: (x_1, \dots, x_q)$, ${\bf y} =: (y_1, \dots, y_q)$, and ${\bf z} =: (z_1, \dots, z_q)$. 
Then, for every $i$, 
$d(x_i, y_i) \le u$ and $d(y_i, z_i) \le v$. 
By the definition of the basic triangle function,  $d(x_i, z_i) \le \Phi_d (u,v)$. 
By the definition of $d_{\infty}$, $d_{\infty} ({\bf x}, {\bf z}) \le  \Phi_d (u,v)$. 
By the definition of the basic triangle function, $\Phi_{d_{\infty}} (u,v) \le  \Phi_d (u,v)$. 

By this inequality, we see that $d_{\infty}$ is a normal or regular semi-metric on $X^q$ if $d$ is a normal or regular semi-metric on $X$, respectively.

Assume that $({\bf x}_n)_n$ is a Cauchy sequence on $(X^q, d_{\infty})$. 
Let ${\bf x}_n =: (x_{n,1}, \dots, x_{n,q}) \in X^q$. 
Then, by the definition of $d_{\infty}$, it holds that for each $i$, $(x_{n,i})_n$ is also a Cauchy sequence on $(X, d)$. 
Since $(X,d)$ is complete, for each $i$, there exists $x_i \in X$ such that $\lim_{n \to \infty} d(x_{n,i}, x_i) = 0$. 
Let ${\bf x} := (x_1, \dots, x_q) \in X^q$. 
Then, by the definition of $d_{\infty}$,
$\lim_{n \to \infty} d_{\infty}({\bf x}_{n}, {\bf x}) = 0$. 
\end{proof}

We secondly deal with the product space of $\mathcal{F}(X)$. 
Let 
\[ d_{HP}^{\infty}({\bf A}, {\bf B}) := \max_{1 \le i \le q} d_{HP}(A_i, B_i), \ {\bf A} = (A_1, \dots, A_q), {\bf B} = (B_1, \dots, B_q) \in \mathcal{F}(X)^q.   \]
Then, 
\begin{Lem}\label{lem:product-HP-semi-metric}
(i) If $(X,d)$ is a semi-metric space, then, $d_{HP}^{\infty}$ is a symmetric pre-metric on $\mathcal{F}(X)^q$. \\
(ii) If a semi-metric space $(X,d)$ satisfies the normal property and the regular property, 
then, $(\mathcal{F}(X)^q, d_{HP}^{\infty})$ satisfies the normal property and the regular property. 
\end{Lem}

\begin{proof}
Assertion (i) follows from the definition of $d_{HP}^{\infty}$.  

We show (ii). 
Let ${\bf A} = (A_1, \dots, A_q), {\bf B} = (B_1, \dots, B_q),  {\bf C} = (C_1, \dots, C_q) \in \mathcal{F}(X)^q$. 
Then, for every $i$, 
$$ d_{HP}(A_i, C_i) \le \Phi_{d_{HP}}\left(d_{HP}(A_i, B_i), d_{HP}(B_i, C_i)\right) \le \Phi_{d_{HP}}\left(d_{HP}^{\infty}({\bf A}, {\bf B}), d_{HP}^{\infty}({\bf B}, {\bf C})\right). $$
Hence, 
$d_{HP}^{\infty}({\bf A}, {\bf C}) \le \Phi_{d_{HP}}\left(d_{HP}^{\infty}({\bf A}, {\bf B}), d_{HP}^{\infty}({\bf B}, {\bf C})\right)$. 
So, 
$\Phi_{d_{HP}^{\infty}}(u,v) \le \Phi_{d_{HP}}(u,v)$. 
By Lemma \ref{lem:HP-normal-and-regular}, 
$(\mathcal{F}(X), d_{HP})$ satisfies the normal property and the regular property. 
Thus we have the normal property and the regular property for $(\mathcal{F}(X)^q, d_{HP}^{\infty})$. 
\end{proof}

We now deal with the graph-directed Markov system. 
By Lemma \ref{lem:weak-contraction-semimet-HP}, 
\begin{Lem}\label{lem:weak-contraction-HP-product}
Let $T$ be a map as in \eqref{eq:def-T}.  
Let  $E := \cup_{i,j \in V} E_{i,j}$ and $\phi := \max_{e \in E} \phi_{e}$. 
Then, $T(\mathcal{F}(X)^q) \subset \mathcal{F}(X)^q$, and, 
\[ d_{HP}^{\infty}(T({\bf A}), T({\bf B})) \le \phi\left( d_{HP}^{\infty}({\bf A}, {\bf B}) \right), \ \ {\bf A}, {\bf B}  \in \mathcal{F}(X)^q, \]
and $\phi$ is a right-continuous comparison function. 
\end{Lem}

We remark that in general it does not hold that $T\left(\overline{\mathcal{F}}(X)^q\right) \subset \overline{\mathcal{F}}(X)^q$. 
However, we can show that $T(\mathcal{K}(X)^q) \subset \mathcal{K}(X)^q \subset \overline{\mathcal{F}}(X)^q$.

\section{Proof}\label{sec:proof}

Throughout this section, $(X,d)$ is assumed to be a normal, regular, and complete semi-metric space. 

\begin{Lem}\label{lem:1-subinv-exist}
There exists an element ${\bf H}_0 \in \mathcal{F}(X)^q$ such that ${\bf H}_0 \ne \emptyset$ and  ${\bf H}_0 \subset T({\bf H}_0)$. 
\end{Lem}

\begin{proof}
By the assumption, there exists a map $J : V \to V$ such that $E_{i, J(i)} \ne \emptyset$. 
Take a map $f_{e_i}, e_i \in E_{i, J(i)}$. 
We recall that each $f_{e_i}$ is a $\phi$-contraction. 
By Lemma \ref{lem:product-semimetric}, $(X^q, d_{\infty})$ is a normal, regular and complete semi-metric space. 
Define a map $F: X^q \to X^q$  by 
\[ F(x_1, \dots, x_q) := \left(f_{e_1}(x_{J(1)}), \dots, f_{e_q} (x_{J(q)})\right). \]  
Since 
\[ d_{\infty}\left(F({\bf x}), F({\bf y})\right) = \max_{i \in V} d(f_{e_i}(x_{J(i)}), f_{e_i}(y_{J(i)})) \le \max_{i \in V}  \phi(d(x_{J(i)}, y_{J(i)})) \] 
\[ \le \phi\left(d_{\infty}({\bf x}, {\bf y})\right), \ {\bf x} = (x_1, \dots, x_q), {\bf y} = (y_1, \dots, y_q), \]
$F$ is a $\phi$-contraction on $(X^q, d_{\infty})$. 
Hence, by Theorem \ref{thm:Bessenyei-Pales-FP}, 
there exists a unique fixed point of $F$. 

Let $(z_1, \dots, z_q)$ be a fixed point of $F$. 
Then, 
$f_{e_i}(z_{J(i)}) = z_i$ for each $i \in V$. 
Let ${\bf H}_0 := \left(\{z_1\}, \dots, \{z_q\}\right) \in \mathcal{F}(X)^q$. 
This is a product of singleton sets. 
Then, 
\[  T({\bf H}_0) = \left(\bigcup_{j = 1}^{q} \bigcup_{e \in E_{i,j}} f_e (\{z_j\})\right)_{i = 1}^{q}. \]
It holds that for each $i \in V$, 
\[ \bigcup_{j = 1}^{q} \bigcup_{e \in E_{i,j}} f_e (\{z_j\}) \supset \bigcup_{e \in E_{i,J(i)}} f_e (\{z_{J(i)}\}) \supset f_{e_i}(\{z_{J(i)}\}) = \{z_i\}. \]
Hence, 
$T({\bf H}_0) \supset {\bf H}_0$. 
\end{proof}

\begin{Lem}\label{lem:2-bdd-inv-exist}
There exists an element ${\bf H} \in \mathcal{F}(X)^q$ such that  $T({\bf H}) = {\bf H}$. 
\end{Lem}

We will show that the union $\cup_n T^n ({\bf H}_0)$ is $T$-invariant. 

\begin{proof}
Let ${\bf H}_{0}$ be as in the proof of Lemma \ref{lem:1-subinv-exist}. 
Let ${\bf H}_{n} := T^n ({\bf H}_{0})$ for $n \ge 1$. 
Then, we see that ${\bf H}_{n} \in \mathcal{F}(X)^q$ by induction on $n$. 
By Lemma \ref{lem:product-semimetric},  
$(\mathcal{F}(X)^q, d_{HP}^{\infty})$ is a normal and regular symmetric pre-metric space. 
By Lemma \ref{lem:weak-contraction-HP-product}, 
$T : \mathcal{F}(X)^q \to \mathcal{F}(X)^q$ is a $\phi$-contraction on $(\mathcal{F}(X)^q, d_{HP}^{\infty})$. 
By Lemma \ref{lem:Cauchy-bounded-normal-regular}, 
$({\bf H}_{n})_n$ is a Cauchy sequence in $(\mathcal{F}(X)^q, d_{HP}^{\infty})$ and hence it is bounded. 
Hence, 
there exists $r_0 > 0$ such that 
$d_{HP}^{\infty}\left({\bf H}_{n}, {\bf H}_{0}\right) < r_0$ for every $n \ge 1$.

It holds that $T$ is an {\it isotone} map, that is, 
$T({\bf A}) \subset T({\bf B})$ holds if ${\bf A} \subset {\bf B}$. 
Therefore, ${\bf H}_{n+1} = T( {\bf H}_{n}) \supset {\bf H}_{n}$ for $n \ge 0$.
Let ${\bf H}_{n} =: (H_{n,1}, \dots, H_{n,q})$ for $n \ge 0$. 
Then, for each $i \in V$, 
$(H_{n,i})_n$ is an increasing sequence of subsets of $X$ and furthermore, $H_{n,i} \subset H_{0,i}(r_0)$ for every $n \ge 0$. 
Let $H_i := \cup_{n \ge 0} H_{n,i}$. 
Then, $H_i \subset H_{0,i}(r_0)$ and $H_i \ne \emptyset$. 
Since $H_{0i}$ is bounded, by Lemma \ref{lem:conseq-normal-semi} (ii), $H_{0,i}(r_0)$ is also bounded. 
Hence, $H_i$ is a bounded subset of $X$. 
Thus, $H_i \in \mathcal{F}(X)$.  
Let ${\bf H} := (H_{1}, \dots, H_{q})$ and $T({\bf H}) =: \left(T({\bf H})_1, \dots, T({\bf H})_q\right)$. 
Then, by \eqref{eq:def-T}, 
\[ T({\bf H})_i = \bigcup_{j = 1}^{q} \bigcup_{e \in E_{i,j}} f_e (H_j) = \bigcup_{j = 1}^{q} \bigcup_{e \in E_{i,j}} f_e \left(\bigcup_{n \ge 0} H_{n,j} \right)  \]
\[ = \bigcup_{n \ge 0} \bigcup_{j = 1}^{q} \bigcup_{e \in E_{i,j}} f_e (H_{n,j} ) = \bigcup_{n \ge 0} H_{n+1, i} = H_i. \]
Hence, $T({\bf H}) = {\bf H}$. 
\end{proof}

\begin{Rem}
In the proof in \cite[Theorem 1]{BessenyeiPenzes2022}, it is stated that the union is bounded. 
This is correct, however, there was no proof in \cite{BessenyeiPenzes2022}. 
The above proof is one way to fix it, and the strategy is similar to the proof in \cite[Theorem 1]{BessenyeiPenzes2021}.  
\end{Rem}

For the Kuratowski measure of non-compactness, let 
$$ \chi^{\infty}({\bf A}) := \max_{i \in V} \chi(A_i), \  {\bf A} = (A_1, \dots, A_q) \in \mathcal{F}(X). $$ 

\begin{Lem}\label{lem:3-bdd-inv-Kuratowski-0}
If ${\bf H} \in \mathcal{F}(X)^q$ and  $T({\bf H}) = {\bf H}$, then, $\chi^{\infty}({\bf H}) = 0$. 
\end{Lem}

\begin{proof}
Let ${\bf H} =: (H_1, \dots, H_q)$. 
Assume $\chi^{\infty}({\bf H}) > 0$. 
Then, for some $i \in V$, 
$\chi(H_i) = \chi^{\infty}({\bf H})  > 0$. 
Since $T({\bf H}) = {\bf H}$, by \eqref{eq:def-T}, \eqref{eq:Kuratowski-union}, and Lemma \ref{lem:condense}, 
\[ \chi(H_i) = \chi\left( \bigcup_{j = 1}^{q} \bigcup_{e \in E_{i,j}} f_e (H_j) \right) \le \max_{e \in \cup_{j} E_{i,j}} \chi(f_e (H_j)) < \max_{j} \chi(H_j) = \chi(H_i).  \]
This is a contradiction. 
\end{proof}

\begin{Lem}\label{lem:4-exist-inv-cpt}
Assume that ${\bf H} \in \mathcal{F}(X)^q$ and  $T({\bf H}) = {\bf H}$. 
Let ${\bf \overline{H}} := (\overline{H_1}, \dots, \overline{H_q})$, where ${\bf H} = (H_1, \dots, H_q)$. 
Then, each $\overline{H_i}$ is compact in $X$ and $T\left({\bf \overline{H}}\right) = {\bf \overline{H}}$.   
\end{Lem}

\begin{proof}
We show that each $\overline{H_i}$ is compact in $X$. 
By Lemma \ref{lem:3-bdd-inv-Kuratowski-0}, $\chi(H_i) = 0$, that is, $H_i$ is totally-bounded for each $i \in V$. 
Since $(X,d)$ is complete, by Lemma \ref{lem:zero-Kuratowski-implies-rel-cpt}, 
$H_i$ is relatively compact in $X$. 

We first show that $T\left({\bf \overline{H}}\right) \subset {\bf \overline{H}}$. 
Let $y = (y_1, \dots, y_q) \in T\left({\bf \overline{H}}\right)$. 
Then, for each $i \in V$, 
there exist $j \in V$ and $e \in E_{i,j}$ such that $y_i \in f_e (\overline{H_j})$. 
Then, there exists a sequence $(x_{k})_k$ in $H_j$ such that $d(y_i, f_e (x_{k})) \to 0, k \to \infty$. 
Since $T({\bf H}) = {\bf H}$, $f_e (H_j) \subset H_i$. 
Hence, $y_i \in \overline{H_i}$, and then, $y \in {\bf \overline{H}}$. 

We second show that $T\left({\bf \overline{H}}\right) \supset {\bf \overline{H}}$. 
Let $T\left({\bf \overline{H}}\right) =: \left(T\left({\bf \overline{H}}\right)_1, \dots, T\left({\bf \overline{H}}\right)_q \right)$. 
By \eqref{eq:def-T}, 
\[ T\left({\bf \overline{H}}\right)_i = \bigcup_{j = 1}^{q} \bigcup_{e \in E_{i,j}} f_e (\overline{H_j}).  \] 
Since $\overline{H_j}$ is compact and $f_e$ is continuous, $f_e (\overline{H_j})$ is also compact. 
Since the union on the right hand side is a finite union, 
$T\left({\bf \overline{H}}\right)_i$ is also compact in $X$. 
By \cite[Lemma 1]{BessenyeiPales2017}, $X$ is Hausdorff. 
Hence,  $T\left({\bf \overline{H}}\right)_i$ is closed in $X$. 
Since $T({\bf H}) = {\bf H}$,  $H_i \subset T\left({\bf \overline{H}}\right)_i$. 
Thus we see that $\overline{H_i} \subset T\left({\bf \overline{H}}\right)_i$. 
\end{proof}

Now we consider the uniqueness. 
Recall the definition of $\overline{\mathcal{F}}(X)$. 

\begin{Lem}\label{lem:5-unique-inv-cpt}
There exists at most one ${\bf H} \in \overline{\mathcal{F}}(X)^q$ such that $T({\bf H}) = {\bf H}$. 
\end{Lem}

\begin{proof}
Assume that ${\bf H}_{i} \in \overline{\mathcal{F}}(X)^q$ such that $T({\bf H}_{i}) = {\bf H}_{i}$, $i = 1,2$. 
By Lemma \ref{lem:weak-contraction-HP-product}, 
\[ d_{HP}^{\infty} \left( {\bf H}_{1}, {\bf H}_{2}\right) = d_{HP}^{\infty} \left( T({\bf H}_{1}), T({\bf H}_{2})\right) \le \phi\left( d_{HP}^{\infty} \left( {\bf H}_{1}, {\bf H}_{2}\right) \right). \]
Therefore, 
if $d_{HP}^{\infty} \left( {\bf H}_{1}, {\bf H}_{2}\right) > 0$, then, $d_{HP}^{\infty} \left( {\bf H}_{1}, {\bf H}_{2}\right) < d_{HP}^{\infty} \left( {\bf H}_{1}, {\bf H}_{2}\right)$. 
Hence, 
$d_{HP}^{\infty} \left( {\bf H}_{1}, {\bf H}_{2}\right) = 0$. 
Since ${\bf H}_{i} \in \overline{\mathcal{F}}(X)^q$, $i=1,2$, 
by Lemma \ref{lem:HP-is-semi-metric-if-closed}, 
${\bf H}_{1} = {\bf H}_{2}$. 
\end{proof}

By Lemma \ref{lem:4-exist-inv-cpt} and Lemma \ref{lem:5-unique-inv-cpt}, 
we have Theorem \ref{thm:main-construct-inv-GDM}.

\section{Transform for semi-metric}\label{sec:trans-semi-met}

If $d = d(x,y)$ is a metric on a set $X$, then, $d(x,y)^{\alpha}$ is also a metric on $X$ if $0 < \alpha < 1$. 
This means that $\Psi(x) = x^{\alpha}$ transforms a metric on a set into another metric on the same set. 
We apply this consideration to the case of semi-metrics. 
We say that a map $\Psi : [0,\infty) \to [0, \infty)$ is a {\it semi-metric transformer} 
if $\Psi(0) = 0$ and $\Psi$ is increasing on $[0,\infty)$ and $\Psi(t) > 0$ for every $t > 0$,  and $\Psi$ is continuous at $0$.

\begin{Lem}
Let $X$ be a non-empty set. 
If $d$ is a semi-metric on $X$ and $\Psi$ is a semi-metric transformer, then, 
$\widetilde{d} (x,y) := \Psi \circ d (x,y) = \Psi(d(x,y))$ is also a semi-metric on $X$, and, the topology defined by $d$ is identical with the topology defined by $\widetilde{d}$.  
\end{Lem}

We denote a ball $B(x,r)$ by $B_{d}(x,r)$ or $B_{\widetilde{d}}(x,r)$, according to the metric.  

\begin{proof}
Since $\Phi(x) = 0$ if and only if $x=0$, 
$\widetilde{d}(x,y) = 0$ if and only if $x = y$. 
The symmetry is obvious. 
Therefore, $\widetilde{d}$ is also a semi-metric on $X$. 

Assume that $U$ is an open set with respect to $d$. 
Let $x \in U$. 
Then, there exists $r > 0$ such that $B_d (x,r) \subset U$. 
Let $y \in B_{\widetilde{d}}(x, \Psi(r))$. 
Then,  $\widetilde{d}(x,y) = \Psi(d(x,y)) < \Psi(r)$. 
Since $\Psi$ is increasing, $d(x,y) < r$, that is, $y \in B_d (x,r)$. 
Hence, $B_{\widetilde{d}}(x, \Psi(r)) \subset B_d (x,r) \subset U$. 
Therefore, $U$ is an open set with respect to $\widetilde{d}$.

Conversely, assume that $U$ is an open set with respect to $\widetilde{d}$. 
Let $x \in U$. 
Then, there exists $r > 0$ such that $B_{\widetilde{d}} (x,r) \subset U$. 
Since $\Psi(0) = 0$ and $\Psi$ is continuous at $0$, 
there exists $\delta > 0$ such that $\Psi(\delta) < r$. 
Then, $B_{d}(x, \delta) \subset B_{\widetilde{d}} (x,r) \subset U$. 
Therefore, $U$ is an open set with respect to $d$. 
\end{proof}

\begin{Exa}
The following functions are semi-metric transformers.\\ 
(i) $\Psi(t) := t^{\alpha}$, $\alpha > 0$.\\
(ii) $\Psi(t) = \min\{t^{\alpha},1\}$, $\alpha > 0$.\\
(iii)  $\Psi(t) := \left(\frac{t}{1+t}\right)^{\alpha}$, $\alpha > 0$. 
The semi-metric in \cite[Example 2.2]{DungHang2017} is obtained if $\alpha = 4$ and $d$ is the Euclidean distance. \\
(iv) $\Psi(t)$ is the Cantor function on $t \in [0,1]$ and $\Psi(t) = 1$ on $t \in (1,\infty)$. 
\end{Exa}

\begin{Lem}
Let $X$ be a non-empty set. 
Let $\Psi$ be a semi-metric transformer. 
If $d$ is a normal or regular or complete semi-metric on $X$, then, 
$\widetilde{d} := \Psi \circ d$ is also a normal or regular or complete semi-metric on $X$, respectively.  
\end{Lem}

For $u \ge 0$, let 
\[ \Psi^{(-1)}(u) := \inf\{t \ge 0 | \Psi(t) > u\}, \]
where we let $\inf \emptyset := \infty$. 
This is a generalized inverse of $\Psi$. 
If $\Psi$ is strictly increasing, then, there exists the inverse map $ \Psi^{-1}$ of $\Psi$. 
However, $\Psi$ can be non-continuous or unbounded,  the domain of $ \Psi^{-1}$ may not be $[0,\infty)$. 
Hence, the generalized inverse is introduced. 
Let $\displaystyle \Psi(\infty) := \lim_{u \to \infty} \Psi(u)$. 
Since $\Psi$ is increasing and $\Psi(t) > 0$ for $t > 0$, $\Psi(\infty)$ exists and  is contained in $(0,\infty]$. 

\begin{proof}
We denote the basic triangle function for $\widetilde{d}$ by $\Phi_{\widetilde{d}}$.  
Let $u, v \in [0, \Psi(\infty))$. 
Assume that $\widetilde{d}(x,z) \le u$ and $\widetilde{d}(y,z) \le v$. 
Since $\Psi$ is increasing, 
$d(x,z) \le \Psi^{(-1)}(u) < \infty$ and $d(y,z) \le \Psi^{(-1)}(v) < \infty$. 
Hence, $d(x,y) \le \Phi_d \left( \Psi^{(-1)}(u), \Psi^{(-1)}(v)\right)$. 
Thus, 
\begin{equation}\label{eq:basic-triangle-ineq-transformer}
\Phi_{\widetilde{d}} (u,v) \le \Psi \left(\Phi_d \left( \Psi^{(-1)}(u), \Psi^{(-1)}(v)\right)\right), \ u, v \in [0,\Psi(\infty)).
\end{equation} 

Assume that $d$ is normal. 
If $\Psi(\infty) < \infty$, then, $\Phi_{\widetilde{d}} (u,v) \le \textup{diam}(X) \le \Psi(\infty)< \infty$ for every $u, v \in [0, \infty)$. 
If $\Psi(\infty) = \infty$, then, $\Phi_d \left( \Psi^{(-1)}(u), \Psi^{(-1)}(v)\right) < \infty$ for every $u, v \in [0, \infty)$, 
and hence, $\Phi_{\widetilde{d}} (u,v) < \infty$ for every $u, v \in [0, \infty)$. 
Therefore, $\widetilde{d}$ is normal. 

Let $\epsilon > 0$. 
Then, by the assumption, $\Psi(\epsilon) > 0$.  
By the definition of $ \Psi^{(-1)}$,  for every $u \in (0, \Psi(\epsilon))$, $ \Psi^{(-1)}(u) \le \epsilon$. 
Thus we see that 
\begin{equation}\label{eq:gen-inv-conv-zero}
\lim_{u \to +0} \Psi^{(-1)}(u) = 0.
\end{equation} 

Assume that $d$ is regular. 
Then, 
$\displaystyle \lim_{(u,v) \to (0,0)} \Phi_d \left( \Psi^{(-1)}(u), \Psi^{(-1)}(v)\right) = 0$. 
Since $\Psi$ is continuous at $0$, 
$\displaystyle \lim_{(u,v) \to (0,0)}  \Psi \left(\Phi_d \left( \Psi^{(-1)}(u), \Psi^{(-1)}(v)\right)\right) = 0$. 
By this and \eqref{eq:basic-triangle-ineq-transformer}, 
$\displaystyle \lim_{(u,v) \to (0,0)} \Phi_{\widetilde{d}} (u,v) = 0$. 
Thus we see that $\widetilde{d}$ is regular.

Assume that $d$ is complete. 
Assume that $\displaystyle \lim_{m,n \to \infty} \widetilde{d} (x_m, x_n) = 0$. 
By the definition of $ \Psi^{(-1)}$, $u \le \Psi^{(-1)}(\Psi(u))$ for every $u \ge 0$, 
and hence, 
\[ d(x_m, x_n) \le \Psi^{(-1)}\left(\Psi(d(x_m, x_n))\right) = \Psi^{(-1)}\left(\widetilde{d}(x_m, x_n)\right). \]
By this and \eqref{eq:gen-inv-conv-zero}, 
$\displaystyle \lim_{m,n \to \infty} d (x_m, x_n) = 0$. 
Since $d$ is complete, there exists $x \in X$ such that $\displaystyle \lim_{n \to \infty} d(x_n, x) = 0$. 
Since $\Psi$ is continuous at $0$, $\displaystyle \lim_{n \to \infty} \widetilde{d}(x_n, x) = 0$. 
Thus we see that $\widetilde{d}$ is complete. 
\end{proof}

\begin{Rem}
(i) The assumption that $\Psi$ is continuous at $0$ is not used in the statement of normal property. 
The normal property concerns the {\it global} property of semi-metric.\\
(ii) It is easy to see that $\Psi \circ d$ is uniformly equivalent to $d$, 
so we can apply a general result in \cite[Theorem 3.2]{ChrzaszczJachymskiTurobos2018} to the regular property. 
The regular property concerns the {\it local} property of semi-metric.
\end{Rem}

\begin{Exa}\label{exa:1}
Let $X = [0,1]$ and $d(x,y) := |x-y|^{\alpha}$, where $|x-y|$ is the Euclidean distance and $\alpha > 0$. 
Then, $d$ is a regular, normal and complete semi-metric on $X$. 
We consider a graph-directed set on $X$. 
Let $q = 2$. 
Assume that $|E_{i,j}| = 1$ for every $i, j \in \{1,2\}$. 
We denote $f_{e}$ for $e \in E_{i,j}$ by $f_{i,j}$. 
Let 
\[ f_{1,1}(x) = \frac{x}{7}, \ f_{1,2}(x) = \frac{x+2}{7}, \ f_{2,1}(x) = \frac{x+4}{7}, \   f_{2,2}(x) = \frac{x+6}{7}. \]
Then, for ${\bf H} = (H_1, H_2)$, $T({\bf H}) = {\bf H}$ holds if and only if 
\begin{equation}\label{eq:ex1-rep}
(H_1, H_2) = (f_{1,1}(H_1) \cup f_{1,2}(H_2), f_{2,1}(H_1) \cup f_{2,2}(H_2)). 
\end{equation}

We define finite sets $A_n$ and $B_n$, $n \ge 1$, inductively. 
Let $A_0 := \{0\}$ and $B_0 := \{1\}$. 
Let ${\bf H}_0 := (A_0, B_0)$. 
Since $f_{1,1}(0) = 0$ and $f_{2,2}(1) = 1$, 
${\bf H}_0 \subset T({\bf H}_0)$. 
Let 
\[ g_{1,1} (x) := \frac{x}{7}, \ g_{1,2}(x) := \frac{3-x}{7}, \ g_{2,1}(x) = \frac{5-x}{7}, \  g_{2,2}(x) := \frac{x+6}{7}.   \]
 Let $A_{n+1} := g_{1,1}(A_n) \cup g_{1,2}(A_n)$ and $B_{n+1} = g_{2,1}(B_n) \cup g_{2,2}(B_n)$ for $n \ge 0$.  
 Let $r(x) := 1-x$. 

We show that $r(A_n) = B_n$ by induction on $n$. 
The case that $n=0$ is obvious. 
Assume that $r(A_n) = B_n$. 
Then, since $r(r(x)) = x$, $r \circ g_{1,1} \circ r = g_{2,2}$, $r \circ g_{1,2} \circ r = g_{2,1}$, 
it holds that 
$A_n = r(B_n)$ and 
$$ r(A_{n+1}) = r(g_{1,1}(A_n)) \cup r(g_{1,2}(A_n)) = r(g_{1,1}(r(B_n))) \cup r(g_{1,2}(r(B_n)))$$
$$ = g_{2,2}(B_n) \cup g_{2,1}(B_n) = B_{n+1}. $$

We show that $T^n ({\bf H}_0) = (A_n, B_n)$ by induction on $n$. 
The case that $n=0$ is obvious. 
Assume that $T^n ({\bf H}_0) = (A_n, B_n)$. 
Then, by \eqref{eq:ex1-rep}, 
\[  T^{n+1} ({\bf H}_0) = T(A_n, B_n) = (f_{1,1}(A_n) \cup f_{1,2}(B_n), f_{2,1}(A_n) \cup f_{2,2}(B_n)). \]
Since $f_{1,1} = g_{1,1}$, $f_{2,2} = g_{2,2}$, $g_{1,2} = f_{1,2} \circ r$ and $g_{2,1} = f_{2,1} \circ r$, 
\[  (f_{1,1}(A_n) \cup f_{1,2}(B_n), f_{2,1}(A_n) \cup f_{2,2}(B_n)) =  (g_{1,1}(A_n) \cup g_{1,2}(A_n), g_{2,1}(B_n) \cup g_{2,2}(B_n)) \]
\[= (A_{n+1}, B_{n+1}). \]

Let $K_i, i = 1,2$, be non-empty compact subsets of $X = [0,1]$ such that $K_i = g_{i,1}(K_i) \cup g_{i,2}(K_i), i = 1,2$.  
Let ${\bf K} := (K_1, K_2)$. 
Then, 
\begin{equation}\label{eq:inv-set-eq-union-ex}  
T({\bf K}) = {\bf K} = (f_{1,1}(K_1) \cup f_{1,2}(K_2), f_{2,1}(K_1) \cup f_{2,2}(K_2)) = \overline{\cup_{n} T^n ({\bf H}_0)}. 
\end{equation}
We see that $r(K_1) = K_2$. 

Let $S(H) := \cup_{i,j \in \{1,2\}} f_{i,j}(H)$. 
Let $K$ be a unique non-empty compact subset of $[0,1]$ such that $K = S(K)$. 
Then, $K$ is a generalized Cantor set, and in terms of the base-7 expansion, we obtain that 
\[ K = \left\{ \sum_{i=1}^{\infty} \frac{a_i}{7^i} \middle| a_i \in \{0,2,4,6\}, i \ge 1  \right\}. \]

We remark that 
\[ f_{1,1}([0,1]) \subset \left[0, \frac{1}{7}\right], \ f_{1,2}([0,1]) \subset \left[\frac{2}{7}, \frac{3}{7}\right], \  f_{2,1}([0,1]) \subset \left[\frac{4}{7}, \frac{5}{7}\right], \ f_{2,2}([0,1]) \subset \left[\frac{6}{7}, 1\right], \]
 and $f_{i,j}([0,1])$ are disjoint. 
Hence, $K_1 \subset [0,3/7]$ and $K_2 \subset [4/7,1]$.  
In particular, $K_1$ and $K_2$ are disjoint. 
Since each $f_{i,j}$ is injective, 
$f_{i,j}(K_j) \subsetneq f_{i,j}(K_1 \cup K_2) \subset f_{i,j}([0,1])$. 
Therefore, 
\[ K_1 \cup K_2 = f_{1,1}(K_1) \cup f_{1,2}(K_2) \cup f_{2,1}(K_1) \cup f_{2,2}(K_2)  \subsetneq S(K_1 \cup K_2).  \]
By the argument in Section \ref{sec:proof} for $q=1$,
$\overline{\cup_{n \ge 0} S^n (K_1 \cup K_2)} = K$, and hence, $K_1 \cup K_2 \subsetneq K$. 
\end{Exa}

In Example \ref{exa:1}, 
each $f_{i,j}$ is an affine function. 
Now we give an example such that all $f_{i,j}$ are  non-affine functions. 

\begin{Exa}\label{exa:2}
Let $X = [0,1]$ and $d(x,y) := |x-y|$. 
We consider a graph-directed set on $X$. 
Let $q = 2$. 
Assume that $|E_{i,j}| = 1$ for every $i, j \in \{1,2\}$. 
We denote $f_{e}$ for $e \in E_{i,j}$ by $f_{i,j}$. 
Let 
\[ f_{1,1}(x) = \frac{x}{6x+1}, \, f_{1,2}(x) = \frac{-x^2+2x+4}{7}, \, f_{2,1}(x) = \frac{x^2+2}{7}, \,  f_{2,2}(x) = \frac{6-5x}{7-6x}. \]
Then, each $f_{i,j}$ is a weak contraction in the sense of Matkowski-Rus. 
Let 
$g_{1,1} := f_{1,1}$, $g_{2,2} := f_{2,2}$, $g_{1,2} := f_{1,2} \circ r$, and $g_{2,1} := f_{2,1} \circ r$. 
Then, $r \circ g_{1,1} \circ r = g_{2,2}$, $r \circ g_{1,2} \circ r = g_{2,1}$. 

We define finite sets $A_n$ and $B_n$, $n \ge 1$, in the same manner as in Example \ref{exa:1}. 
We see that $T^n ({\bf H}_0) = (A_n, B_n)$ for $n \ge 0$. 
Let $K_i, i = 1,2$, be non-empty compact subsets of $X = [0,1]$ such that $K_i = g_{i,1}(K_i) \cup g_{i,2}(K_i), i = 1,2$. 
Let ${\bf K} := (K_1, K_2)$. 
Then, we can show that \eqref{eq:inv-set-eq-union-ex}. 

Let $S(H) := \cup_{i,j \in \{1,2\}} f_{i,j}(H)$. 
Let $K$ be a unique non-empty compact subset of $[0,1]$ such that $K = S(K)$. 
Since $K_1 \subset [0,1/7] \cup [4/7, 5/7]$ and $K_2 \subset [2/7, 3/7] \cup [6/7,1]$, 
$K_1$ and $K_2$ are disjoint. 
We can show that $K_1 \cup K_2 \subsetneq K$. 
\end{Exa}

We remark that $K_1 = K_2 = K$ can happen. 

\begin{Exa}\label{exa:3}
We keep the same notation as in Example \ref{exa:1}. 
Let 
\[ f_{1,1}(x) = \frac{x}{x+1}, \, f_{1,2}(x) = \frac{x+1}{2}, \, f_{2,1}(x) = \frac{x}{2}, \,  f_{2,2}(x) = \frac{1}{2-x}. \]
Then, each $f_{i,j}$ is a weak contraction in the sense of Matkowski-Rus. 
We see that 
\[ f_{1,1}([0,1]) = f_{2,1} ([0,1]) = \left[0, \frac{1}{2}\right], \ f_{1,2}([0,1]) = f_{2,2} ([0,1]) = \left[\frac{1}{2}, 1\right]. \]
Hence, $T({\bf K}) = {\bf K}$, where we let $K_1 = K_2 = [0,1]$ and ${\bf K} := (K_1, K_2)$. 
Furthermore, $S([0,1]) = [0,1]$. 
\end{Exa}

We can give many examples. 
\cite[Example 3]{MauldinWilliams1988} gives an example different from the above.

\section{Graph-directed de Rham functional equation}\label{sec:fe}

In this section, we generalize the framework of the de Rham functional equation to the graph-directed setting. 
Functional equations indexed by a directed network were recently considered in \cite{MuramotoSekiguchi2020}. 
Here we deal with the functional equation on $[0,1]$ indexed by a graph-directed finite graph in terms of the graph-directed invariant set and conjugate equations \cite{BuescuSerpa2021, Okamura2019}. 
For simplicity we focus on the case corresponding to $q=2$.

We say that a doubly-indexed family of functions  $h_{i,j}, i, j \in \{1,2\}$, on $[0,1]$ is a {\it compatible family} with comparison function $\phi$ 
if each $h_{i,j}$ is strictly increasing and a $\phi$-contraction and a compatibility condition holds, specifically, 
\begin{equation}\label{eq:compatibility-def} 
0 = h_{i,1} (0) < h_{i,1}(1) = h_{i,2}(0) < h_{i,2}(1) = 1, \ i = 1,2. 
\end{equation} 

We state several  properties of the compatible family $h$.  
Let 
\[ D^h_{i, n}  := \left\{  h_{i,i_1} \circ h_{i_1,i_2} \circ \cdots \circ h_{i_{n-1},i_n} (x_n) \middle| i_1, \dots, i_n \in \{1,2\}, x_n \in \{0,1\}  \right\} \]
and let $D^h_{i} := \cup_{n \ge 1} D^h_{i, n}$. 
Then, $(D^h_{i, n})_n$ is increasing with respect to $n$. 
Let 
\[ I^h_{i} (i_1, \dots, i_n)  := \left[ h_{i,i_1} \circ h_{i_1,i_2} \circ \cdots \circ h_{i_{n-1},i_n} (0), h_{i,i_1} \circ h_{i_1,i_2} \circ \cdots \circ h_{i_{n-1},i_n} (1) \right]. \]
Let $\mathcal{I}^h_{i,n} := \left\{ I^h_{i} (i_1, \dots, i_n) : i_1, \dots, i_n \in \{1,2\} \right\}$. 
By the compatibility condition \eqref{eq:compatibility-def}, 
\begin{equation}\label{eq:interval-split} 
I^h_{i} (i_1, \dots, i_n) = I^h_{i} (i_1, \dots, i_n,1) \cup I^h_{i} (i_1, \dots, i_n,2). 
\end{equation}

For every $x \in [0,1]$, 
there exists a sequence $(i_n)_n$ in $\{1,2\}$ such that 
$x \in  I^h_{i} (i_1, \dots, i_n)$ for each $n$, and, if $x \notin D^h_{i}$, then, such $(i_n)_n$ is unique. 
We can choose $i_n$ inductively. 
Indeed, by \eqref{eq:interval-split}, 
if $x \in  I^h_{i} (i_1, \dots, i_n)$, then, $x \in  I^h_{i} (i_1, \dots, i_n, 1)$ or $x \in  I^h_{i} (i_1, \dots, i_n, 2)$.
If additionally $x \notin D^h_{i}$, then, exactly one of $x \in  I^h_{i} (i_1, \dots, i_n, 1)$ and $x \in  I^h_{i} (i_1, \dots, i_n, 2)$ holds.

We denote the length of an interval $I$ by $|I|$. 
Let 
\[ \Delta^{h}_{i, n} := \max_{i_1, \dots, i_n \in \{1,2\}}  \left| I^h_{i} (i_1, \dots, i_n) \right|, \ n \ge 1, \  i = 1,2.  \]
Since $\left| I^h_{i} (i_1, \dots, i_n) \right| \le \phi^n (1)$ and $\phi$ is a comparison function, 
\begin{equation}\label{eq:interval-shrink} 
\lim_{n \to \infty} \Delta^{h}_{i, n} = 0. 
\end{equation}
Hence, $D^h_{i}$ is dense in $[0,1]$. 

Let $\{f_{i,j}\}_{i,j}$ and $\{g_{i,j}\}_{i,j}$ be two compatible families on $[0,1]$ with a common comparison function $\phi$ under the Euclidean distance\footnote{We can replace the Eucliean metric with a different metric  by the semi-metric transformer in the above section, as long as $\{f_{i,j}\}_{i,j}$ and $\{g_{i,j}\}_{i,j}$ are contractions in the sense of Matkowski-Rus under the new metric.}.  
For each $i = 1,2$, 
we define functions $\varphi_i$, on $D^f_{i}$ such that for $i_1, \dots, i_n \in \{1,2\}, x_n \in \{0,1\}$, 
\[ \varphi_i \left( f_{i, i_1} \circ f_{i_1, i_2} \circ \cdots \circ f_{i_{n-1}, i_n} (x_n) \right) = g_{i, i_1} \circ g_{i_1, i_2} \circ \cdots \circ g_{i_{n-1}, i_n} (x_n). \]
We show that this is well-defined. 
It suffices to show that if 
\begin{equation}\label{eq:assumption-well-defined}
f_{i, i_1} \circ f_{i_1, i_2} \circ \cdots \circ f_{i_{n-1}, i_n} (x_n) =  f_{i, i_1^{\prime}} \circ f_{i_1^{\prime}, i_2^{\prime}} \circ \cdots \circ f_{i_{n-1}^{\prime}, i_n^{\prime}} (x_n^{\prime}),  
\end{equation}
for $i_1, \dots, i_n, i_1^{\prime}, \dots, i_n^{\prime} \in \{1,2\}$ and $x_n, x_{n}^{\prime} \in \{0,1\}$, 
then, 
\begin{equation}\label{eq:conclusion-well-defined}  
g_{i, i_1} \circ g_{i_1, i_2} \circ \cdots \circ g_{i_{n-1}, i_n} (x_n) =  g_{i, i_1^{\prime}} \circ g_{i_1^{\prime}, i_2^{\prime}} \circ \cdots \circ g_{i_{n-1}^{\prime}, i_n^{\prime}} (x_n^{\prime}). 
\end{equation}

Assume that $i_k = i_{k}^{\prime} $ for every $k \in \{1,\dots, n\}$. 
By \eqref{eq:assumption-well-defined} and the fact that $f_{i, i_1} \circ f_{i_1, i_2} \circ \cdots \circ f_{i_{n-1}, i_n}$ is strictly increasing, 
$x_n = x_n^{\prime}$ and hence \eqref{eq:conclusion-well-defined} holds. 

Assume that $i_k \ne i_{k}^{\prime} $ for some $k \in \{1,\dots, n\}$. 
Let $\ell := \min\left\{k \in \{1,\dots, n\} | i_k \ne i_{k}^{\prime} \right\}$.
Let $i_0 := i$. 
Then, by \eqref{eq:assumption-well-defined} and the fact that $f_{i, i_1} \circ f_{i_1, i_2} \circ \cdots \circ f_{i_{\ell-2}, i_{\ell-1}}$ is strictly increasing, 
\[  f_{i_{\ell-1}, i_{\ell}} \circ f_{i_{\ell}, i_{\ell+1}} \circ \cdots \circ f_{i_{n-1}, i_n} (x_n) =  f_{i_{\ell -1}, i_{\ell}^{\prime}} \circ f_{i_{\ell}^{\prime}, i_{\ell+1}^{\prime}} \circ \cdots \circ f_{i_{n-1}^{\prime}, i_n^{\prime}} (x_n^{\prime}). \]
Without loss of generality, we can assume that $i_{\ell} < i_{\ell}^{\prime}$. 
Then, $i_{\ell} = 1$ and $i_{\ell}^{\prime} = 2$. 
Since $f_{i_{\ell-1}, 1}$ and $f_{i_{\ell-1}, 2}$ are strictly increasing, 
\[ f_{i_{\ell-1}, 1}(1) \ge f_{i_{\ell-1}, 1} \circ f_{1, i_{\ell+1}} \circ \cdots \circ f_{i_{n-1}, i_n} (x_n) = f_{i_{\ell -1}, 2} \circ f_{2, i_{\ell+1}^{\prime}} \circ \cdots \circ f_{i_{n-1}^{\prime}, i_n^{\prime}} (x_n^{\prime}) \ge f_{i_{\ell -1}, 2} (0). \]
By this and the compatibility condition \eqref{eq:compatibility-def}  for $\{f_{i,j}\}$, 
\[ f_{i_{\ell-1}, 1}(1) = f_{i_{\ell-1}, 1} \circ f_{1, i_{\ell+1}} \circ \cdots \circ f_{i_{n-1}, i_n} (x_n) = f_{i_{\ell -1}, 2} \circ f_{2, i_{\ell+1}^{\prime}} \circ \cdots \circ f_{i_{n-1}^{\prime}, i_n^{\prime}} (x_n^{\prime}) = f_{i_{\ell -1}, 2} (0). \]
Hence, 
\[ f_{1, i_{\ell+1}} \circ f_{i_{\ell+1}, i_{\ell + 2}} \circ \cdots \circ f_{i_{n-1}, i_n} (x_n) = 1, \, f_{2, i_{\ell+1}^{\prime}} \circ f_{i_{\ell+1}, i_{\ell + 2}} \circ \cdots \circ f_{i_{n-1}^{\prime}, i_n^{\prime}} (x_n^{\prime}) = 0.\]
In the same manner, 
we see that $(i_k, j_k) = (2,1)$ for $k \in \{\ell+1, \dots, n\}$ by induction on $k$, and $(x_n, x_n^{\prime}) = (1,0)$. 
By the compatibility condition \eqref{eq:compatibility-def}  for $\{g_{i,j}\}$, 
it holds that 
\[ g_{1, i_{\ell+1}} \circ \cdots \circ g_{i_{n-1}, i_n} (x_n) = 1, \ g_{2, i_{\ell+1}^{\prime}} \circ \cdots \circ g_{i_{n-1}^{\prime}, i_n^{\prime}} (x_n^{\prime}) = 0, \]
and furthermore, 
\[ g_{i_{\ell}-1,1} \circ g_{1, i_{\ell+1}} \circ \cdots \circ g_{i_{n-1}, i_n} (x_n) = g_{i_{\ell}-1,1} (1) = g_{i_{\ell}-1,2} (0) =  g_{i_{\ell}-1,2} \circ g_{2, i_{\ell+1}^{\prime}} \circ \cdots \circ g_{i_{n-1}^{\prime}, i_n^{\prime}} (x_n^{\prime}). \]
Thus we have \eqref{eq:conclusion-well-defined}.

Now we define $\varphi_i (x)$ for $x \notin D^{f}_{i}$. 
Take a unique sequence $(i_n)_n$ in $\{1,2\}$ such that $x \in I^f_{i} (i_1, \dots, i_n)$ for every $n \ge 1$. 
Then, 
$\cap_{n \ge 1} I^g_{i} (i_1, \dots, i_n)$ is a compact subset of $[0,1]$ and 
by \eqref{eq:interval-shrink}, 
it is a singleton.
Let $y \in [0,1]$ such that $\{y\} = \cap_{n \ge 1} I^g_{i} (i_1, \dots, i_n)$. 
Then, we define $\varphi_i (x) := y$. 
Thus a map $\varphi_i : [0,1] \to [0,1]$ is defined. 

On $[0,1]$, we put the topology induced by the Euclidean metric. 

\begin{Prop}\label{prop:basic-sol-gd-dR}
Let $i \in \{1,2\}$. 
Then, it holds that \\
(i) $\varphi_i$ is increasing on $[0,1]$. \\
(ii) $\varphi_i$ is continuous on $[0,1]$.\\
(iii) A pair of maps $(\varphi_1, \varphi_2)$ satisfies the following system of functional equations: 
\begin{equation}\label{eq:gd-dR-fe-def} 
g_{i,j} \circ \varphi_j = \varphi_i \circ f_{i,j}, \ i, j \in \{1,2\}.  
\end{equation}
\end{Prop}

\eqref{eq:gd-dR-fe-def} can be expressed as 
\begin{equation}\label{eq:gd-dR-fe-def-alt} 
\begin{cases} g_{1,1}(\varphi_{1}(x)) = \varphi_1 (f_{1,1}(x)) & x \in [0,1] \\ g_{1,2}(\varphi_{2}(x)) = \varphi_1 (f_{1,2}(x))  & x \in [0,1] \\ g_{2,1}(\varphi_{1}(x)) = \varphi_2 (f_{2,1}(x))   & x \in [0,1] \\ g_{2,2}(\varphi_{2}(x)) = \varphi_2 (f_{2,2}(x))  & x \in [0,1] \end{cases}. 
\end{equation}

\begin{proof}
(i) 
Let $x < y$, and 
\[ N_i := \sup\left\{n \middle| \textup{ there exists }  I \in \mathcal{I}^{f}_{i,n} \textup{ such that } x, y \in I \right\}. \]
By \eqref{eq:interval-shrink}, $N_i$ is finite. 
Let $x, y \in I^{f}_{i} (i_1, \dots, i_{N_i})$. 
Then, 
$x \in I^{f}_{i} (i_1, \dots, i_{N_i}, 1)$ and $y \in I^{f}_{i} (i_1, \dots, i_{N_i},2)$. 
By the definition of $\varphi$, 
\[ \varphi(x) \le g_{i,i_1} \circ g_{i_1, i_2} \circ \cdots \circ g_{i_{N_i},1} (1) =  g_{i,i_1} \circ g_{i_1, i_2} \circ \cdots \circ g_{i_{N_i},2} (0) \le \varphi(y). \]

(ii) By (i), it suffices to show that for each $x \in (0,1)$, 
\begin{equation}\label{eq:right-left-limit}
\lim_{t \to x + 0} \varphi(t) = \lim_{t \to x - 0} \varphi(t), 
\end{equation} 
$\lim_{t \to + 0} \varphi(t) = 0$, and $\lim_{t \to 1-0} \varphi(t) = 1$. 

Assume that $x \notin D^h_{i}$. 
Then, there exists a unique sequence $(i_n)_n$ in $\{1,2\}$ such that $x \in I^f_{i} (i_1, \dots, i_n)$ for every $n \ge 1$. 
Then, $\varphi (x) \in  I^g_{i} (i_1, \dots, i_n)$ for every $n \ge 1$. 
Hence, for every $n \ge 1$, 
\[ f_{i, i_1} \circ f_{i_1, i_2} \circ \cdots \circ f_{i_{n-1}, i_n} (0) < x < f_{i, i_1} \circ f_{i_1, i_2} \circ \cdots \circ f_{i_{n-1}, i_n} (1),  \]
\[ g_{i, i_1} \circ g_{i_1, i_2} \circ \cdots \circ g_{i_{n-1}, i_n} (0) \le \varphi (x) \le g_{i, i_1} \circ g_{i_1, i_2} \circ \cdots \circ g_{i_{n-1}, i_n} (1), \]
and 
\[ \varphi_i (f_{i, i_1} \circ f_{i_1, i_2} \circ \cdots \circ f_{i_{n-1}, i_n} (x_n) ) = g_{i, i_1} \circ g_{i_1, i_2} \circ \cdots \circ g_{i_{n-1}, i_n} (x_n), \ x_n \in \{0,1\}.  \]
By these inequalities, (i),  and \eqref{eq:interval-shrink}, 
\eqref{eq:right-left-limit} holds. 

Assume that $x \in D^h_{i} \cap (0,1)$. 
Then, there exists $n \ge 1$ such that $x \in D^h_{i,n} \setminus D^h_{i,n-1}$, where we let $D^h_{i,0} := \emptyset$. 
Then, there exist  $i_1, \dots, i_{n-1} \in \{1,2\}$ such that  
\[ x =  f_{i, i_1} \circ f_{i_1, i_2} \circ \cdots \circ f_{i_{n-1}, 1} (1) =  f_{i, i_1} \circ f_{i_1, i_2} \circ \cdots \circ f_{i_{n-1}, 2} (0). \]
where  $x = f_{i,1}(1) = f_{i,2}(0)$ for $n = 1$. 
Hence, 
\[  f_{i, i_1} \circ f_{i_1, i_2} \circ \cdots \circ f_{i_{n-1}, 1}(0) < x < f_{i, i_1} \circ f_{i_1, i_2} \circ \cdots \circ f_{i_{n-1}, 2} (1), \]
and 
\[ g_{i, i_1} \circ g_{i_1, i_2} \circ \cdots \circ g_{i_{n-1}, 1} (0) \le \varphi (x) \le g_{i, i_1} \circ g_{i_1, i_2} \circ \cdots \circ g_{i_{n-1}, 2} (1). \]
By these inequalities, (i),  and \eqref{eq:interval-shrink}, 
\eqref{eq:right-left-limit} holds. 

If $i_n = 1$ for every $n \ge 1$, then,  $h_{i, i_1} \circ h_{i_1, i_2} \circ \cdots \circ h_{i_{n-1}, i_n}(0) = 0$ for $h = f \textup{ or } g$. 
By this, (i),  and \eqref{eq:interval-shrink}, $\lim_{t \to + 0} \varphi(t) = 0$. 

If $i_n = 2$ for every $n \ge 1$, then,  $h_{i, i_1} \circ h_{i_1, i_2} \circ \cdots \circ h_{i_{n-1}, i_n}(1) = 1$ for $h = f \textup{ or } g$. 
By this, (i), and \eqref{eq:interval-shrink}, $\lim_{t \to 1- 0} \varphi(t) = 1$.

(iii) Let $x \in  D^{f}_{j}$. 
Then, there exist $j_1, \dots, j_n \in \{1,2\}, y_n \in \{0,1\}$ such that  $x = f_{j,j_1} \circ \cdots \circ f_{j_{n-1},j_{n}}(y_n)$. 
Then, by the definition of $\varphi_i, i = 1,2$,  
\[ g_{i,j} \circ \varphi_j (x) = g_{i,j} \circ \varphi_j  \left(f_{j,j_1} \circ \cdots \circ f_{j_{n-1},j_{n}}(y_n)\right) = g_{i,j} \circ g_{j,j_1} \circ \cdots \circ g_{j_{n-1},j_{n}}(y_n)  \]
\[ = \varphi_i \left( f_{i,j} \circ f_{j,j_1} \circ \cdots \circ f_{j_{n-1},j_{n}}(y_n) \right) =  \varphi_i \circ f_{i,j} (x). \]
By the contraction assumption and (ii), all of $g_{i,j}, \varphi_j, \varphi_i$ and $f_{i,j}$ are continuous. 
Since $D^{f}_{j}$ is dense in $[0,1]$, 
we see that $g_{i,j} \circ \varphi_j (x)  =  \varphi_i \circ f_{i,j} (x)$ for every $x \in [0,1]$. 
\end{proof}

We remark that in the same manner, we can also show that $\varphi_i$ is strictly increasing. 

Second, we state a connection between the functional equation and the graph-directed set. 
Put the product topology on $[0,1]^2$. 
Let 
\[ \Phi_{i,j}(x,y) := (f_{i,j}(x), g_{i,j}(y)), \ \ x, y \in [0,1], \ i = 1,2. \] 
Let 
\begin{equation}\label{eq:def-T-dR-graph} 
T(H_1, H_2) := \left(  \Phi_{1,1}(H_1) \cup  \Phi_{1,2}(H_2), \Phi_{2,1}(H_1) \cup  \Phi_{2,2}(H_2)  \right), \ H_1, H_2 \in \mathcal{P}([0,1]). 
\end{equation}
Then, $(0,0)$ is a fixed point of $\Phi_{1,1}$ and $(1,1)$ is a fixed point of $\Phi_{2,2}$. 
Let ${\bf H}_0 := \left(\{(0,0)\}, \{(1,1)\}\right)$. 
Then, 
\[ T({\bf H}_0) = \left( \{(0,0), (1,1)\},  \{(0,0), (1,1)\} \right),  \]
and hence it is $T$-subinvariant, that is, ${\bf H}_0 \subset T({\bf H}_0)$. 
We also see that 
\[ T^2 ({\bf H}_0) = \left( \{(0,0),  (f_{1,1}(1), g_{1,1}(1)), (1,1)\},  \{(0,0), (f_{2,1}(1), g_{2,1}(1)), (1,1)\} \right).  \]
For $i = 1,2$, let 
\[  D_{i,n}  := \]
{\small\[  \left\{  \left(f_{i, i_1} \circ f_{i_1, i_2} \circ \cdots \circ f_{i_{n-1},i_n} (x_n), g_{i, i_1} \circ g_{i_1, i_2} \circ \cdots \circ g_{i_{n-1},i_n} (x_n)\right) \middle| i_1, \dots, i_n \in \{1,2\}, x_n \in \{0,1\}  \right\}. \]}
Let ${\bf D}_n := (D_{1,n}, D_{2,n})$. 
By induction on $n$, 
we see that $T^{n} \left({\bf H}_0 \right) = {\bf D}_{n-1}$ for $n \ge 2$. 
By the arguments in Section \ref{sec:proof}, 
\[ {\bf K} := \overline{\bigcup_{n \ge 0} T^{n} \left({\bf H}_0 \right)} = \overline{\bigcup_{n \ge 1} {\bf D}_{n}} \]
is $T$-invariant. 

We finally show that $K_i = \left\{(x, \varphi_i (x) | x \in [0,1] \right\}$, $i= 1,2$, where ${\bf K} =: (K_1, K_2)$. 
Let $\widetilde{K_i} := \left\{(x, \varphi_i (x) | x \in [0,1] \right\}$. 
By Proposition \ref{prop:basic-sol-gd-dR} (ii), 
the map $x \mapsto (x, \varphi_i (x))$ is a continuous map from $[0,1]$ to $[0,1]^2$. 
Since $[0,1]$ is compact, 
$\widetilde{K_i}$ is also compact. 
By the definition of $\varphi_i$ on $D^f_i$, 
$D_{i,n} \subset \widetilde{K_i}$ for every $n$.  
Hence, $K_i \subset \widetilde{K_i}$.
By the definition of $\varphi_i$ on $[0,1] \setminus D^f_i$, 
$\widetilde{K_i} \subset \overline{\cup_{n \ge 1} D_{i,n}} = K_i$. 
This completes the proof. 

Discrete approximation is a natural approach for analysis for  \eqref{eq:gd-dR-fe-def} and the framework of Theorem \ref{thm:main-construct-inv-GDM} aligns with the functional equation well. 

\begin{Exa}
Let 
\[ f_{1,1}(x) := \frac{x}{2}, \ f_{1,2}(x) := \frac{x+1}{2}, \ f_{2,1}(x) := \frac{x}{2}, \ f_{2,2}(x) := \frac{x+1}{2},  \]
and 
\[ g_{1,1}(x) := \frac{x}{3}, \ g_{1,2}(x) := \frac{2x+1}{3}, \ g_{2,1}(x) := \frac{x}{4}, \ g_{2,2}(x) := \frac{3x+1}{4}.  \]
Then, $\{f_{i,j}\}$ and $\{g_{i,j}\}$ are compatible families with a certain common comparison function. 
There exists a pair of strictly increasing and continuous functions $(\varphi_1, \varphi_2)$, which is a solution of  \eqref{eq:gd-dR-fe-def}. 
In this case, \eqref{eq:gd-dR-fe-def-alt} is expressed as 
\[ \begin{cases} \frac{1}{3} \varphi_{1}(x) = \varphi_1 \left(\frac{x}{2}\right) & x \in [0,1] \\ \frac{2}{3} \varphi_{2}(x)+ \frac{1}{3} = \varphi_1 \left(\frac{x+1}{2}\right)  & x \in [0,1] \\ \frac{1}{4} \varphi_{1}(x) = \varphi_2 \left(\frac{x}{2}\right)   & x \in [0,1] \\ \frac{3}{4} \varphi_{2}(x) + \frac{1}{4} = \varphi_2 \left(\frac{x+1}{2}\right)  & x \in [0,1] \end{cases}.\]
This can also be rewritten as 
\[ \begin{cases} \varphi_1 (x) = \begin{cases} \frac{1}{3} \varphi_{1}(2x)  & x \in [0,1/2] \\ \frac{2}{3} \varphi_{2}(2x-1)+ \frac{1}{3}  & x \in [1/2,1] \end{cases} \\ \varphi_2 (x) =   \begin{cases} \frac{1}{4} \varphi_{1}(2x)   & x \in [0,1/2]  \\ \frac{3}{4} \varphi_{2}(2x-1)+ \frac{1}{4}  & x \in [1/2,1] \end{cases}  \end{cases}. \]

The graphs of the solution are given as follows: 
\begin{figure}[H]
\begin{minipage}{0.49\columnwidth}
\centering
\includegraphics[scale=0.37, bb = 0 0 640 480]{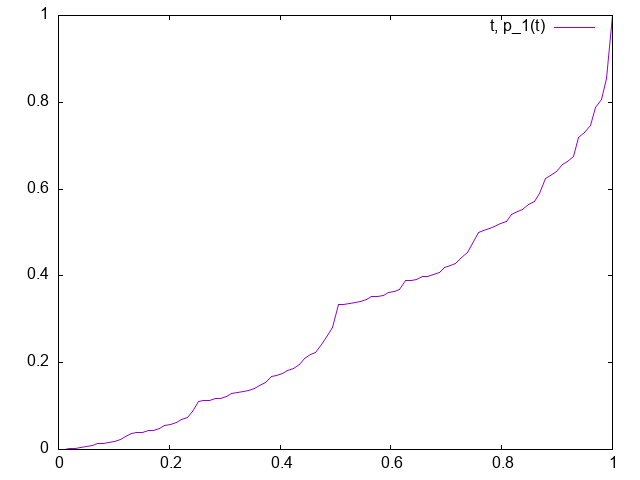}
\caption{Graph of $\varphi_1$}
\end{minipage}
\begin{minipage}{0.49\columnwidth}
\centering
\includegraphics[scale=0.37, bb = 0 0 640 480]{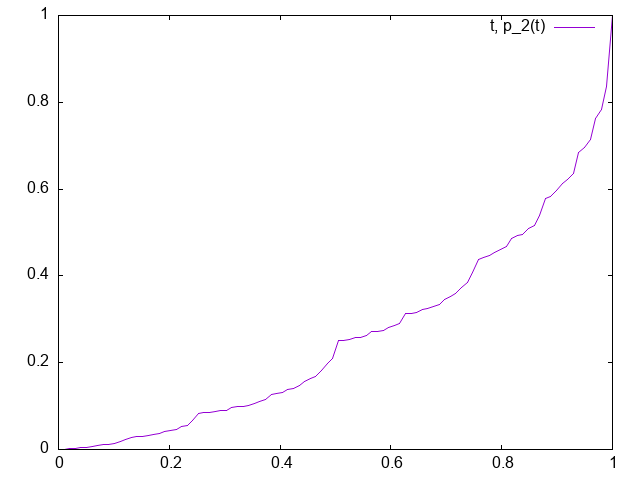}
\caption{Graph of $\varphi_2$}
\end{minipage}
\end{figure}

If we replace $g_{21}$ and $g_{22}$ with non-affine weak contractions
\[ g_{2,1}(x) = \frac{x}{x+1}, \ g_{2,2} (x) = \frac{1}{2-x}, \ x \in [0,1], \]
then, \eqref{eq:gd-dR-fe-def-alt} is expressed as 
\[ \begin{cases} \varphi_1 (x) = \begin{cases} \frac{1}{3} \varphi_{1}(2x)  & x \in [0,1/2] \\ \frac{2}{3} \varphi_{2}(2x-1)+ \frac{1}{3}  & x \in [1/2,1] \end{cases} \\ \varphi_2 (x) =   \begin{cases} \frac{\varphi_{1}(2x) }{\varphi_{1}(2x)  + 1}  & x \in [0,1/2]  \\ \frac{1}{2 - \varphi_{2}(2x-1)}  & x \in [1/2,1] \end{cases}  \end{cases}. \]
The choice of $g_{2,1}$ and $g_{2,2}$ is related to Minkowski's question-mark function. 
The graphs of the solution are given as follows: 
\begin{figure}[H]
\begin{minipage}{0.49\columnwidth}
\centering
\includegraphics[scale=0.37, bb = 0 0 640 480]{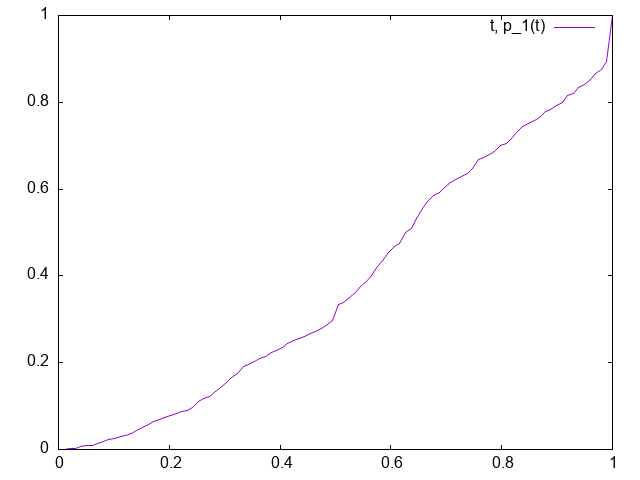}
\caption{Graph of $\varphi_1$}
\end{minipage}
\begin{minipage}{0.49\columnwidth}
\centering
\includegraphics[scale=0.37, bb = 0 0 640 480]{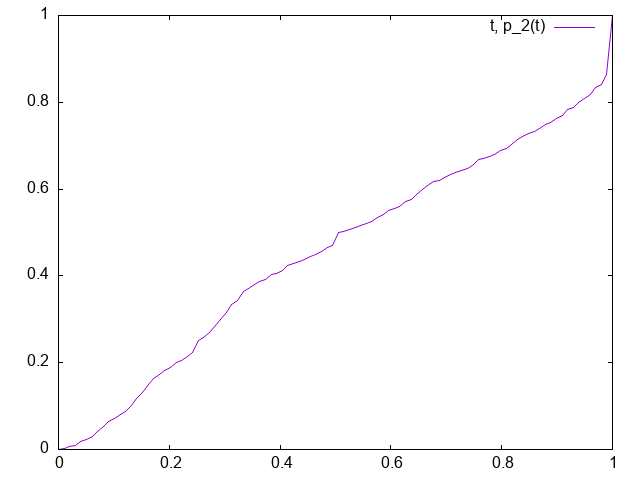}
\caption{Graph of $\varphi_2$}
\end{minipage}
\end{figure}

It would be interesting to know whether $\varphi_1$ and $\varphi_2$ are singular functions or not. 
\end{Exa}

{\it Acknowledgements} \ The author would like to give his thanks to a referee for his or her comments. 
He is grateful to Omer Cantor for comments regarding Remark \ref{rem:closed}.   
He is also grateful to a production editor for improving the English of the manuscript. 

\bibliographystyle{plain}
\bibliography{GDM-simple}

\end{document}